\documentclass[11pt,oneside]{amsart}
\usepackage{amsmath}
\usepackage{amssymb}
\usepackage{amscd}
\usepackage{amsthm}
\usepackage{latexsym}
\usepackage[pdftex]{graphicx}
\usepackage{pinlabel}
\usepackage{enumerate}
\usepackage{epstopdf}

\numberwithin{equation}{section}
\theoremstyle{plain}
\newtheorem*{nonumthmA}{Theorem A}

\newtheorem{thm}{Theorem}[section]

\newtheorem{thm-defn}[thm]{Theorem/Definition}
\newtheorem{lemma}[thm]{Lemma}
\newtheorem{prop}[thm]{Proposition}

\theoremstyle{definition}
\newtheorem{defn}[thm]{Definition}

\theoremstyle{remark}
\newtheorem{rmk}[thm]{Remark}

\newcommand{\A}{{\mathcal A}}

\newcommand{\C}{\mathbb C}

\newcommand{\K}{{\mathcal K}}
\renewcommand{\L}{{\mathcal L}}

\newcommand{\Q}{\mathbb Q}

\newcommand{\R}{\mathbb R}

\newcommand{\Z}{\mathbb Z}

\renewcommand{\O}{\mathcal O}

\newcommand{\oo}{\infty}

\begin{document}

\begin{abstract}
Fin\-tu\-shel and Stern defined the ration\-al blow-down construction \cite{FSrbd} for smooth 4-manifolds, where a linear plumbing configuration of spheres $C_n$ is replaced with a rational homology ball $B_n$, $n \geq 2$. Subsequently, Symington \cite{Sym1} defined this procedure in the symplectic category, where a symplectic $C_n$ (given by symplectic spheres) is replaced by a symplectic copy of $B_n$ to yield a new symplectic manifold. As a result, a symplectic rational blow-down can be performed on a manifold whenever such a configuration of symplectic spheres can be found. In this paper, we define the inverse procedure, the \textit{rational blow-up} in the symplectic category, where we present the symplectic structure of $B_n$ as an entirely standard symplectic neighborhood of a certain Lagrangian 2-cell complex. Consequently, a symplectic rational blow-up can be performed on a manifold whenever such a Lagrangian 2-cell complex is found. 

\end{abstract}

\title[Symplectic Rational Blow-up]{Symplectic Rational Blow-up}
\author{Tatyana Khodorovskiy}

\maketitle

\section{Introduction}
In 1997, Fintushel and Stern \cite{FSrbd} defined the rational blow-down operation for smooth $4$-manifolds, a generalization of the standard blow-down operation. For smooth $4$-manifolds, the standard blow-down is performed by removing a neighborhood of a sphere with self-intersection $(-1)$ and replacing it with a standard $4$-ball $B^4$. The rational blow-down involves replacing a negative definite plumbing $4$-manifold with a rational homology ball. In order to define it, we first begin with a description of the negative definite plumbing $4$-manifold $C_n$, $n \geq 2$, as seen in Figure~\ref{f:cn}, where each dot represents a sphere, $S_i$, in the plumbing configuration. The integers above the dots are the self-intersection numbers of the plumbed spheres: $[S_1]^2 = -(n+2)$ and $[S_i]^2 = -2$ for $2 \leq i \leq n-1$.

\begin{figure}[ht!]
\labellist
\small\hair 2pt
\pinlabel $-(n+2)$ at 30 4.7
\pinlabel $-2$ at 60 4.7
\pinlabel $-2$ at 90 4.7
\pinlabel $-2$ at 173 4.7
\pinlabel $-2$ at 203 4.7
\pinlabel $S_1$ at 32 2.5 
\pinlabel $S_2$ at 62 2.5 
\pinlabel $S_3$ at 92 2.5 
\pinlabel $S_{n-2}$ at 175 2.5 
\pinlabel $S_{n-1}$ at 205 2.5 
\endlabellist
\centering
\includegraphics[width=120mm,height=30mm]{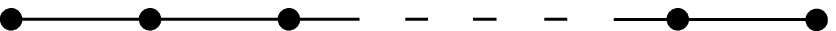}
\caption{{\bf Plumbing diagram of $C_n$, $n\geq2$}}
\label{f:cn}
\end{figure}

The boundary of $C_n$ is the lens space $L(n^2,n-1)$, thus $\pi_1(\partial C_n) \cong H_1(\partial C_n ; \Z) \cong \Z/n^2\Z$. (Note, when we write the lens space $L(p,q)$, we mean it is the $3$-manifold obtained by performing $-\frac{p}{q}$ surgery on the unknot.) This follows from the fact that $[-n-2, -2, \ldots -2]$, with $(n-2)$ many $(-2)$'s is the continued fraction expansion of $\frac{n^2}{1-n}$.

\begin{figure}[ht!]
\labellist
\small\hair 2pt
\pinlabel $n-1$ at -25 240
\pinlabel $n$ at 235 160
\endlabellist
\centering
\includegraphics[scale=0.25]{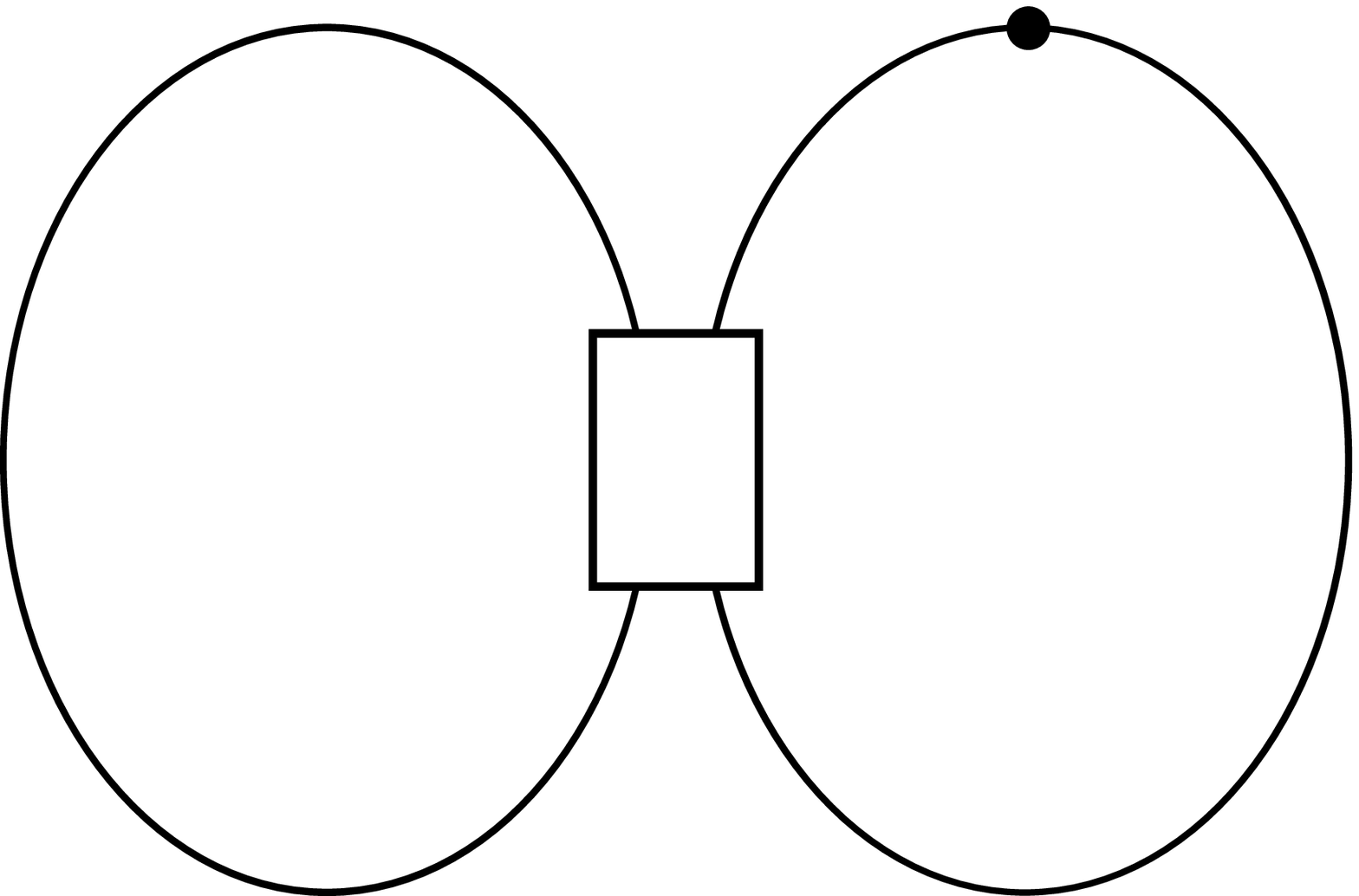}
\caption{{\bf Kirby diagram of $B_n$}}
\label{f:bn}
\end{figure}

Let $B_n$ be the $4$-manifold as defined by the Kirby diagram in Figure~\ref{f:bn} (for a more extensive description of $B_n$, see section~\ref{sec:bn}). The manifold $B_n$ is a rational homology ball, i.e. $H_*(B_n;\Q) \cong H_*(B^4;\Q)$. The boundary of $B_n$ is also the lens space $L(n^2,n-1)$ \cite{CaHa}. Moreover, any self-diffeomorphism of $\partial B_n$ extends to $B_n$ \cite{FSrbd}. Now, we can define the rational blow-down of a $4$-manifold $X$:

\begin{defn}
(\cite{FSrbd}, also see \cite{GS}) Let $X$ be a smooth $4$-manifold. Assume that $C_n$ embeds in $X$, so that $X = C_n \cup_{L(n^2,n-1)} X_0$. The 4-manifold $X_{(n)} = B_n \cup_{L(n^2,n-1)} X_0$ is by definition the \textit{rational blow-down} of $X$ along the given copy of $C_n$.
\end{defn}

Fintushel and Stern \cite{FSrbd} also showed how to compute Seiberg-Witten and Donaldson invariants of $X_{(n)}$ from the respective invariants of $X$. In addition, they showed that certain smooth logarithmic transforms can be alternatively expressed as a series of blow-ups and rational blow-downs. In 1998, (Margaret) Symington \cite{Sym1} proved that the rational blow-down operation can be performed in the symplectic category. More precisely, she showed that if in a symplectic 4-manifold $(M,\omega)$ there is a symplectic embedding of a configuration $C_n$ of symplectic spheres, then there exists a symplectic model for $B_n$ such that the \textit{rational blow-down} of $(M, \omega)$, along $C_n$ is also a symplectic $4$-manifold. (Note, we will often abuse notation and write $C_n$ both for the actual plumbing $4$-manifold and the plumbing configuration of spheres in that $4$-manifold.)

As a result, Symington described when a symplectic $4$-manifold can be symplectically rationally blown down. We would like to investigate the following question: \textbf{when can a symplectic $4$-manifold be symplectically rationally blown up?} By \textit{rational blow-up}, (at least in the smooth category) we mean the inverse operation of \textit{rational blow-down}: if a 4-manifold has an embedded rational homology ball $B_n$, then we can rationally blow it up by replacing the $B_n$ with the negative definite plumbing $C_n$. In order to do that, we first need to verify that rationally blowing up makes sense in the symplectic category. Moreover, we wish to define a ``true" inverse operation to the symplectic rational blow-down. For the symplectic rational blow-down, the existence of a symplectic configuration of spheres $C_n$ in a symplectic manifold makes it possible to perform the operation. In other words, all you need to carry out this procedure is certain ``2-dimensional data" in the symplectic 4-manifold. In the same vein, we will define the symplectic rational blow-up operation, where the 2-dimensional data will be a certain Lagrangian 2-cell complex.

The first step towards such a definition is to equip $B_n$ with a symplectic structure, such that it is the ``standard" symplectic neighborhood of a certain ($2$-dimensional) ``Lagrangian core" $\L_{n,1}$ (see section~\ref{sec:srbuintro} and for an illustration with $n=3$ see Figure~\ref{f:L3}). For $n=2$, $\L_{2,1}$ is simply a Lagrangian $\R P^2$. For $n \geq 3$, $\L_{n,1}$ is a cell complex consisting of an embedded $S^1$ and a $2$-cell $D^2$, whose boundary ``wraps" $n$ times around the embedded $S^1$ (the interior of the $2$-cell $D^2$ is an embedding). Furthermore, the cell complex $\L_{n,1}$ is embedded in such a way that the $2$-cell $D^2$ is Lagrangian. We show, by mirroring the Weinstein Lagrangian embedding theorem, that a symplectic neighborhood of such an $\L_{n,1}$ is entirely standard. As a result, we show that we can obtain a symplectic model for $B_n$ as a standard symplectic neighborhood of this Lagrangian complex $\L_{n,1}$.

Consequently, we prove that a symplectic $4$-manifold $(X,\omega)$ can be symplectically rationally blown up provided there exists this ``Lagrangian core" $\L_{n,1} \subset (X,\omega)$:

\begin{nonumthmA} 
(Theorem~\ref{thm:srbu}) Suppose we can find a ``Lagrangian core" $\L_{n,1} \subset (X,\omega)$, (as in Definition~\ref{d:lnq}), then for some small $\lambda > 0$, there  exists a symplectic embedding of $(B_n,\lambda \omega_n)$ in $(X,\omega)$, and for some $\lambda_0 < \lambda$ and $\mu > 0$, there exists a symplectic 4-manifold $(X',\omega')$ such that $(X',\omega') = ((X,\omega) - (B_n,\lambda_0 \omega_n)) \cup_{\phi} (C_n,\mu \omega'_n)$, where $\phi$ is a symplectic map, and $(B_n,\omega_n)$ and $(C_n,\omega'_n)$ are the symplectic manifolds as defined in section~\ref{sec:sympbncn}. $(X',\omega')$ is called the \textbf{symplectic rational blow-up} of $(X,\omega)$.
\end{nonumthmA}

\noindent In Theorem A above, the scaling coefficient $\lambda$, regulates the ``size" of the rational homology ball $B_n$ that is removed from the symplectic manifold $(X,\omega)$, just like in the definition of the regular symplectic blow-up operation, where one chooses the size of the $4$-ball being removed. The scaling coefficient $\mu$ regulates the symplectic volume of $C_n$ which can ``fit back into" in place of the removed symplectic volume of $B_n$.

The organization of this paper is as follows. In section~\ref{sec:background} we give a detailed description of the rational homology balls $B_n$ and give some background information. In section~\ref{sec:srbumain} we define the ``Lagrangian cores" $\L_{n,1}$ and prove the main theorem. In section~\ref{sec:auxprop}, we prove a proposition used in the proof of main theorem, involving computations of Gompf's invariant for the contact boundaries of the symplectic copies of $B_n$ and $C_n$.

\section{Background}
\label{sec:background}
\subsection{Description of the rational homology balls $B_n$}
\label{sec:bn}
There are several ways to give a description of the rational homology balls $B_n$. One of them is a Kirby calculus diagram seen in Figure~\ref{f:bn}. This represents the following handle decomposition: Start with a 0-handle, a standard 4-disk $D^4$, attach to it a 1-handle $D^1 \times D^3$. Call the resultant space $X_1$, it is diffeomorphic to $S^1 \times D^3$ and has boundary $\partial X_1 = S^1 \times S^2$. Finally, we attach a 2-handle $D^2 \times D^2$. The boundary of the core disk of the 2-handle gets attached to the closed curve, $K$, in $\partial X_1$ which wraps $n$ times around the $S^1 \times \ast$ in $S^1 \times S^2$. We can also represent $B_n$ by a slightly different Kirby diagram, which is more cumbersome to manipulate but is more visually informative, as seen in Figure~\ref{f:bnsph}, where the 1-handle is represented by a pair of balls.

\begin{figure}[ht!]
\labellist
\small\hair 2pt
\pinlabel $n-1$ at 475 18
\pinlabel $\}n$ at 190 65
\endlabellist
\centering
\includegraphics[scale=0.5]{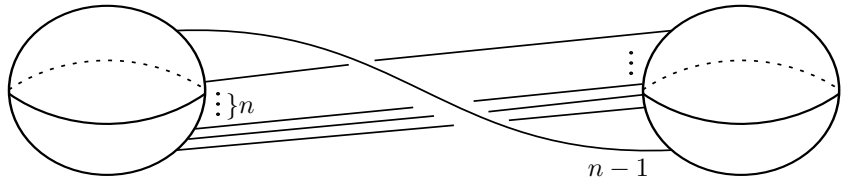}
\caption{{\bf Another Kirby diagram of $B_n$}}
\label{f:bnsph}
\end{figure}

\begin{figure}[ht]
\begin{minipage}[b]{0.45\linewidth}
\centering
\includegraphics[scale=0.30]{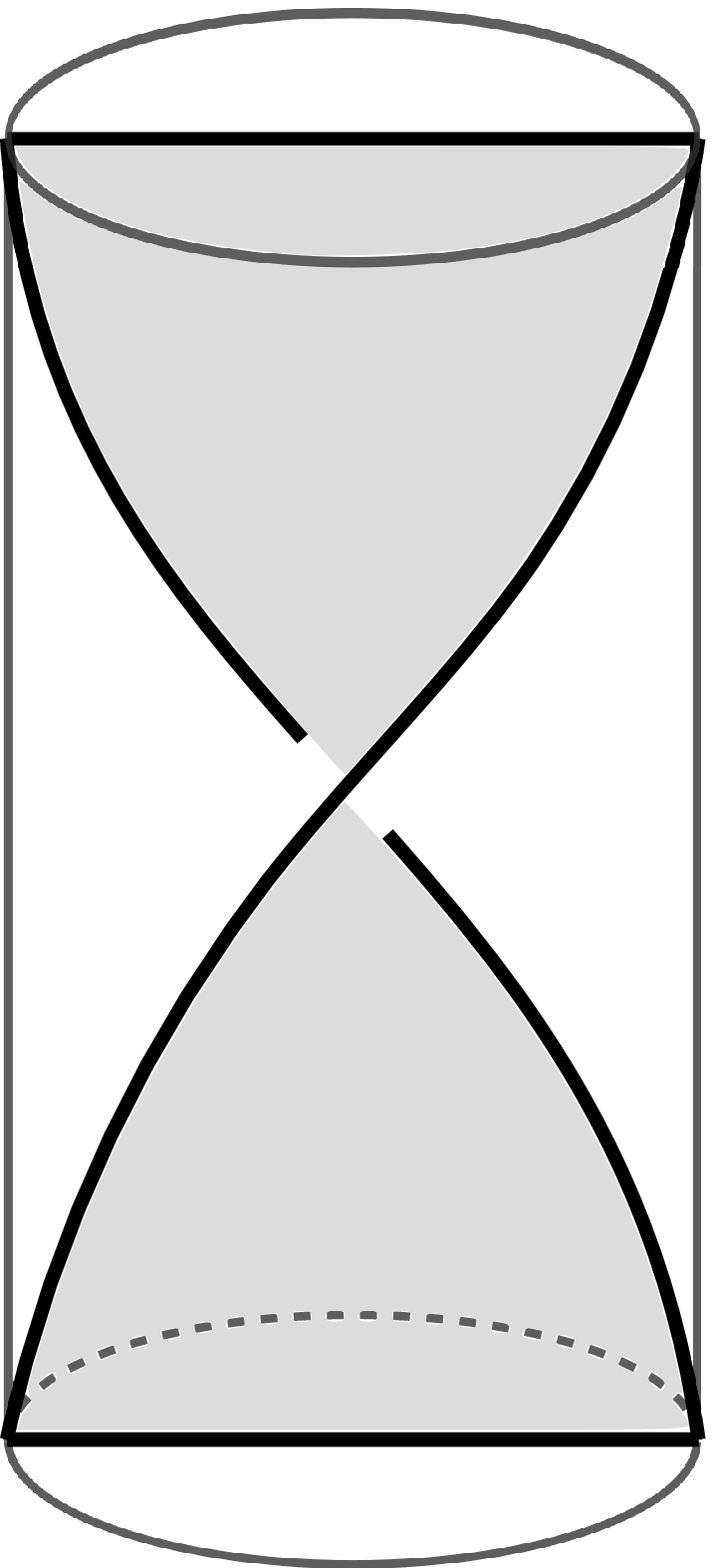}
\caption{$L'_2$}
\label{f:L2}
\end{minipage}
\hspace{0.5cm}
\begin{minipage}[b]{0.45\linewidth}
\centering
\includegraphics[scale=0.30]{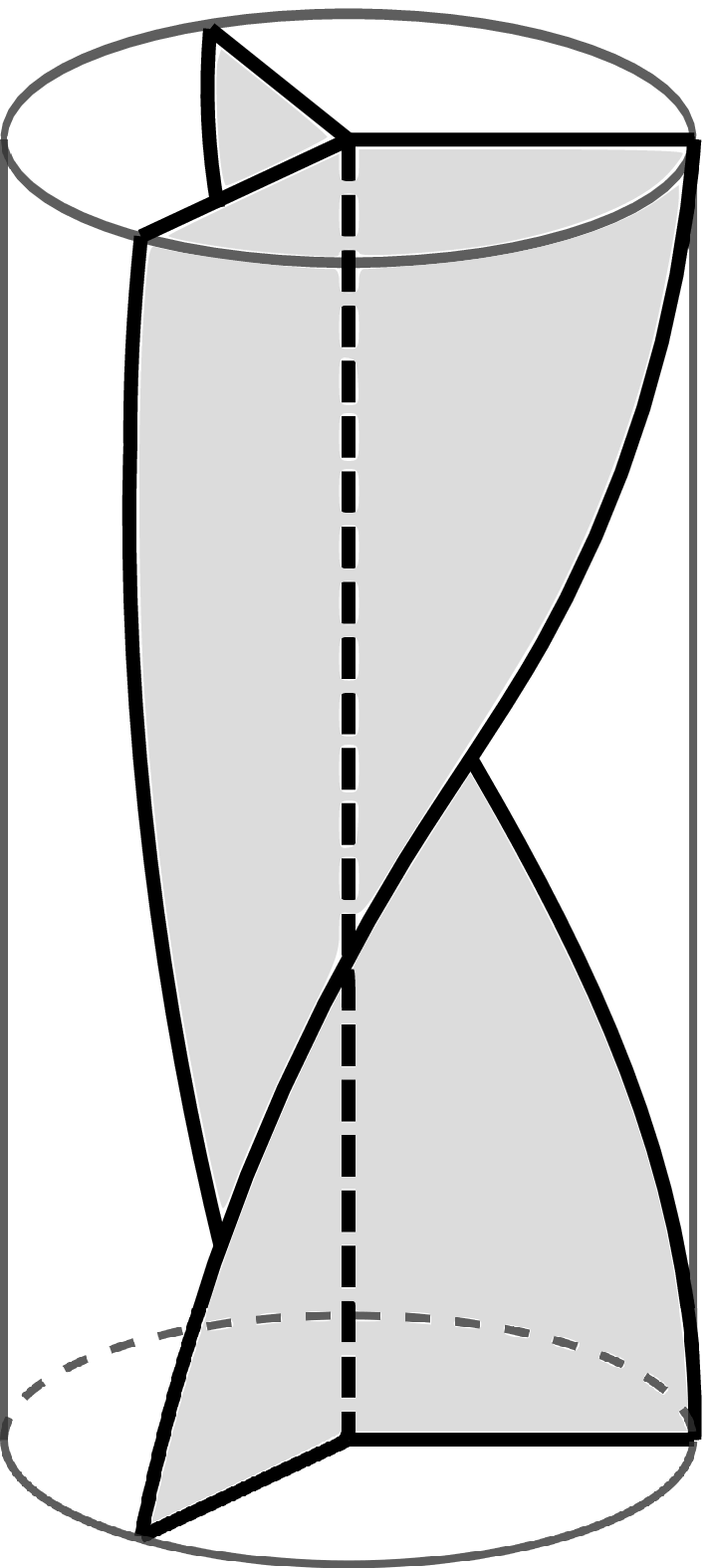}
\caption{$L'_3$}
\label{f:L3}
\end{minipage}
\end{figure}

The rational homology ball $B_2$ can also be described as an unoriented disk bundle over $\mathbb{R}P^2$. Since $\mathbb{R}P^2$ is the union of a Mobius band $M$ and a disk $D$, we can visualize $\mathbb{R}P^2$ sitting inside $B_2$, with the Mobius band and its boundary $(M, \partial M)$ embedded in $(X_1 \cong S^1 \times D^3, \partial X_1 \cong S^1 \times S^2)$ (Figure~\ref{f:L2}, with the ends of the cylinder identified), and the disk $D$ as the core disk of the attaching 2-handle. We will construct something similar for $n\geq 3$. Instead of the Mobius band sitting inside $X_1$, as for $n=2$, we have a ``$n$-Mobius band" (a Moore space), $L'_n$, sitting inside $X_1$. The case of $n=3$ is illustrated in Figure~\ref{f:L3}, again with the ends of the cylinder identified. In other words, $L'_n$ is a singular surface, homotopic to a circle, in $X_1 \cong S^1 \times D^3$, whose boundary is the closed curve $K$ in $\partial X_1 \cong S^1 \times S^2$, and it includes the circle, $S = S^1 \times 0$ in $S^1 \times D^3$. Let $L_n = L'_n \cup_K D$, where $D$ is the core disk of the attached 2-handle (along $K$). We will call $L_n$ the core of the rational homology ball $B_n$; observe, that $L_2 \cong \mathbb{R}P^2$.

The cores $L_n$ will be used as geometrical motivation in the construction of a symplectic structure on the rational homology balls $B_n$. For $n=2$, if we have an embedded $\mathbb{R}P^2$ in $(X,\omega)$, such that $\omega |_{\mathbb{R}P^2} = 0$, (i.e. a Lagrangian $\mathbb{R}P^2$) then the $\mathbb{R}P^2$ will have a totally standard neighborhood, which will be symplectomorphic to the rational homology ball $B_2$. The symplectic structures which we will endow on the rational homology balls $B_n$ will have the cores $L_n \hookrightarrow\ B_n$ be Lagrangian, which we will refer to later as $\L_{n,1}$ in section~\ref{sec:srbumain}.

\subsection{Review of Kirby-Stein calculus}
\label{sec:kscalc}
We will use Eliashberg's Legendrian surgery construction \cite{Eliash} along with Gompf's handlebody constructions of Stein surfaces \cite{Gompf1} to put symplectic structures on the $B_n$s, which will be induced from Stein structures. We will give a brief overview of the aforementioned constructions, beginning with a theorem of Eliashberg's on a 4-manifold admitting a Stein structure \cite{Eliash} \cite{Gompf1}:

\begin{thm}
\label{thm:eli}
A smooth, oriented, open 4-manifold $X$ admits a Stein structure if and only if it is the interior of a (possibly infinite) handlebody such that the following hold:
\begin{enumerate}
\item
Each handle has index $\leq 2$,
\item
Each 2-handle $h_i$ is attached along a Legendrian curve $K_i$ in the contact structure induced on the boundary of the underlying 0- and 1-handles, and
\item
The framing for attaching each $h_i$ is obtained from the canonical framing on $K_i$ by adding a single left (negative) twist.
\end{enumerate}

\noindent A smooth, oriented, compact 4-manifold $X$ admits a Stein structure if and only if it has a handle decomposition satisfying (1), (2), and (3). In either case, any such handle decomposition comes from a strictly plurisubharmonic function (with $\partial X$ a level set).

\end{thm}

From Theorem~\ref{thm:eli}, it follows that if we wanted to construct a Stein surface $S$, such that its strictly plurisubharmonic Morse function did not have any index 1 critical points, then all we have to do to give a handlebody description of $S$ is to specify a Legendrian link $L$ in $S^3 = \partial B^4 = \partial$(0-handle), and attach 2-handles with framing $tb(K_i) - 1$, where $K_i$ are the components the framed link $L$. If we allow index 1 critical points, then we must include 1-handles in the handlebody decomposition of $S$. If a handle decomposition of a compact, oriented 4-manifold has only handles with index 0, 1, or 2, then all that one needs to specify it is a framed link in $\#mS^1 \times S^2 = \partial$ (0-handle $\cup$ 1-handles). Consequently, in order to deal with arbitrary Stein surfaces, Gompf \cite{Gompf1} established a standard form for Legendrian links in $\#mS^1 \times S^2$:

\begin{defn}
\label{d:stdform}
(\cite{Gompf1}, Definition 2.1) A \textit{Legendrian link diagram} in \textit{standard form}, with $m \geq 0$ 1-handles, is given by the following data (see Figure~\ref{f:stdform}):
\begin{enumerate}
\item
A rectangular box parallel to the axes in $\mathbb{R}^2$,
\item
A collection of $m$ distinguished segments of each vertical side of the box, aligned horizontally in pairs and denoted by balls, and
\item
A front projection of a generic Legendrian tangle (i.e. disjoint union of Legendrian knots and arcs) contained in the box, with endpoints lying in the distinguished segments and aligned horizontally in pairs.
\end{enumerate}

\end{defn}

\begin{figure}[ht!]
\labellist
\small\hair 2pt
\pinlabel $\mathbf{Legendrian}$ at 320 500
\pinlabel $\mathbf{tangle}$ at 330 465
\endlabellist
\centering
\includegraphics[scale=0.25]{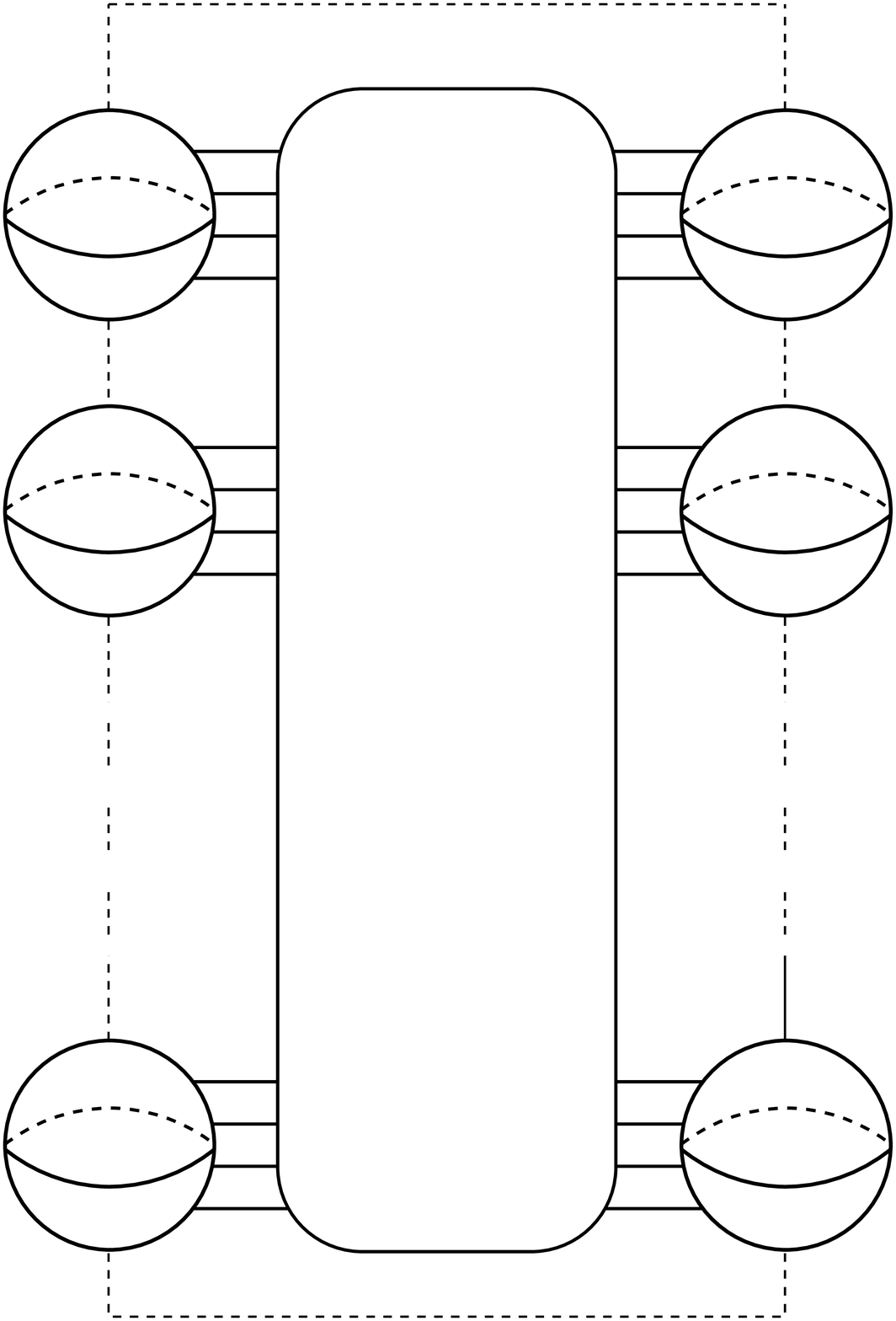}
\caption{{\bf Legendrian link diagram in standard form}}
\label{f:stdform}
\end{figure}

All one needs to do is attach 1-handles to each pair of balls and one gets a link in $\#mS^1 \times S^2$. Using this definition, in Theorem 2.2 in \cite{Gompf1}, Gompf establishes a full list of Kirby-Legendrian calculus type moves that will relate any two such diagrams. More specifically, the theorem states that any two Legendrian links in standard form are contact isotopic in $\partial (\#mS^1 \times S^2)$ if and only if they are related by a sequence of those moves.

The classical invariants of Legendrian knots (see for example \cite{OzSt,Et}), such as the Thurston Bennequin number $tb(K)$ and the rotation number $rot(K)$ still make sense for the Legendrian link diagrams in standard form, although with a few caveats. Both $tb(K)$ and $rot(K)$ can be computed for a knot $K$ that's part of a Legendrian link diagram as in Figure~\ref{f:stdform} from the same formulas as in a standard front projection of Legendrian knots in $\mathbb{R}^3$ (also see Figure~\ref{f:cusps}):
\begin{equation} 
\label{eq:tb}
tb(K) = w(K)-\frac{1}{2}(\lambda(K) + \rho(K)) = w(K) - \lambda(K)
\end{equation}
\begin{equation}
\label{eq:rot}
rot(K) = \lambda_- - \rho_+ = \rho_- - \lambda_+ \, .
\end{equation}

\begin{figure}[ht!]
\labellist
\small\hair 2pt
\pinlabel $\lambda_+$ at 50 10
\pinlabel $\rho_+$ at 150 10
\pinlabel $\lambda_-$ at 250 10
\pinlabel $\rho_-$ at 360 10
\endlabellist
\centering
\includegraphics[height = 30mm, width = 120mm]{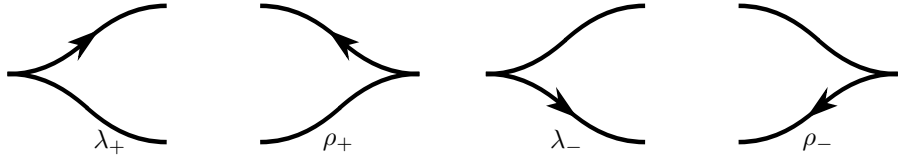}
\caption{Cusps in a front projection of a Legendrian knot}
\label{f:cusps}
\end{figure} 

The invariant $rot(K)$ doesn't change under Gompf's Kirby-Legendrian calculus type moves. However, one of the moves changes $tb(K)$ by twice the number of times (with sign) that $K$ runs over the 1-handle (involved in the move). The change is due to how it is obtained from the diagram and not the canonical framing. Moreover, it is shown that in these Legendrian diagrams (as in Figure~\ref{f:stdform}) the number $tb(K) + rot(K) + 1$ is always congruent modulo $2$ to the number of times that $K$ crosses the 1-handles. 

Putting together the Legendrian link diagrams in standard form, the classical Legendrian knot invariants that can be read from them, the complete list of their Kirby-Legendrian calculus type moves, and Eliashberg's Theorem~\ref{thm:eli}, the following characterization of compact Stein surfaces with boundary can be made:

\begin{prop}
\label{thm:g2}
\cite{Gompf1} A smooth, oriented, compact, connected 4-manifold $X$ admits the structure of a Stein surface (with boundary) if and only if it is given by a handlebody on a Legendrian link in standard form (Definition~\ref{d:stdform}) with the $i$'th 2-handle $h_i$, attached to the $i$'th link component $K_i$, with framing $tb(K_i) - 1$ (as given by Formula~\ref{eq:tb}). Any such handle decomposition is induced by a strictly plurisubharmonic function. The Chern class $c_1(J) \in H^2(X;\mathbb{Z})$ of such a Stein structure $J$ is represented by a cocycle whose value on each $h_i$, oriented as in Theorem~\ref{thm:eli}, is $rot(K_i)$ (as given by Formula~\ref{eq:rot}).
\end{prop} 

The benefits of these Legendrian link diagrams, is that one can compute several useful invariants of the Stein surface and its boundary straight from them. In particular, Gompf (\cite{Gompf1}, section 4) gave a complete set of invariants of 2-plane fields on 3-manifolds, up to their homotopy classes, which, in particular, could be used to distinguish contact structures of the boundaries of Stein surfaces. We will describe one such invariant, $\Gamma$, which we will later use in section~\ref{sec:auxprop}. In general, the classification of 2-plane fields on an oriented 3-manifold $M$ is equivalent to fixing a trivialization of the tangent bundle $TM$ and classifying maps $\varphi : M \rightarrow S^2$ up to homotopy, which was done in \cite{Pont}. $\Gamma$ is an invariant of 2-plane fields on closed, oriented 3-manifolds, that is a 2-dimensional obstruction, thus it measures the associated $spin^c$ structure. The advantage of $\Gamma$ is that it can be specified without keeping explicit track of the choice of trivialization of $TM$, and instead can be measured in terms of spin structures of the 3-manifold $M$.

In order to define $\Gamma$ we need to establish some notation and terminology. Let $(X,J)$ be a Stein surface with a Stein structure $J$. There is a natural way to obtain a contact structure $\xi$ on its boundary $\partial X = M$, by letting $\xi$ be the field of complex lines in $TM \subset TX|_M$, in other words 
\begin{equation*}
\xi = T\partial X \cap JT\partial X \, . 
\end{equation*}

Assume $X$ can be presented in standard form, as in Figure~\ref{f:stdform}. We can construct a manifold $X^*$, which is obtained from $X$ by surgering out all of the 1-handles of $X$ (this can be done canonically). As a result, we have $\partial X = \partial X^* = M$, and $X^*$ can be described by attaching 2-handles along a framed link $L$ in $\partial B^4 = S^3$, which can be obtained by gluing the lateral edges of the box in Figure~\ref{f:stdform}. The 1-handles of $X$ become 2-handles of $X^*$ that are attached along unknots with framing $0$, call this subset of links $L_0 \subset L$. The 2-handles of $X$ remain 2-handles of $X^*$, with the same framing. Since $\Gamma$ will be defined in terms of the spin structures of $M$, it is useful to express the spin structures of $M$ as characteristic sublinks of $L$; thus, for each $\mathfrak{s} \in Spin(M)$, we will associate a characteristic sublink $L(\mathfrak{s}) \subset L$. Recall, that $L'$ is a \textit{characteristic sublink} of $L$ if for each component $K$ of $L$, the framing of $K$ is congruent modulo $2$ to $\ell k(K,L')$ \cite{Gompf1} \cite{GS}. (Note, here $\ell k(A,B)$ is the usual linking number if $A \neq B$, and the framing of $A$ if $A = B$, and is extended bilinearly if $A$ or $B$ have more than one component.) Finally, we can define $\Gamma$ for a boundary of a compact Stein surface, by a formula obtained from a diagram of the Stein surface in standard form:

\begin{thm}
\label{thm:g3}
(Gompf \cite{Gompf1}, Theorem 4.12) Let $X$ be a compact Stein surface in standard form, with $\partial X = (M,\xi)$, and $X^*$, $L = K_1 \cup \ldots \cup K_m$ and $L_0$ as defined above. Let $\{\alpha_1, \ldots, \alpha_m \} \subset H_2(X^*;\mathbb{Z})$ be the basis determined by $\{K_1, \ldots, K_m \}$. Let $\mathfrak{s}$ be a spin structure on $M$, represented by a characteristic sublink $L(\mathfrak{s}) \subset L$. Then $PD\Gamma(\xi,\mathfrak{s})$ is the restriction to $M$ of the class $\rho \in H^2(X^*;\mathbb{Z})$ whose value on each $\alpha_i$ is the integer
\begin{equation}
\label{eq:gamma}
\left\langle \rho, \alpha_i \right\rangle = \frac{1}{2}(rot(K_i) + \ell k(K_i, L_0 + L(\mathfrak{s}))) \, ,
\end{equation}

\noindent (note: $rot(K_i)$ is defined to be $0$ if $K_i \subset L_0$.)     

\end{thm}

\subsection{Description of symplectic structures of $B_n$ and $C_n$}
\label{sec:sympbncn}
First, we will describe the symplectic structure $\omega_n$ on $B_n$, which will be induced from the Stein structure $J_n$. We will present $(B_n, J_n)$ as a Legendrian diagram in standard form, as in Definition~\ref{d:stdform}. However, before that can be done we must first express the $B_n$s with a slightly different Kirby diagram, one that has appropriate framings with which its 2-handles are attached, thus enabling us to put it in Legendrian standard form. Figure~\ref{f:bnsphm} shows another Kirby diagram of $B_n$, that is equivalent to the one in Figure~\ref{f:bn} and Figure~\ref{f:bnsph}, by a series of Kirby moves seen in Appendix~\ref{a:appa}. 

\begin{figure}[ht!]
\labellist
\small\hair 2pt
\pinlabel $-n-1$ at 465 155
\pinlabel $\}n$ at 190 115
\pinlabel $-n-1$ at 190 310
\pinlabel $-n$ at 327 262
\endlabellist
\centering
\includegraphics[scale=0.5]{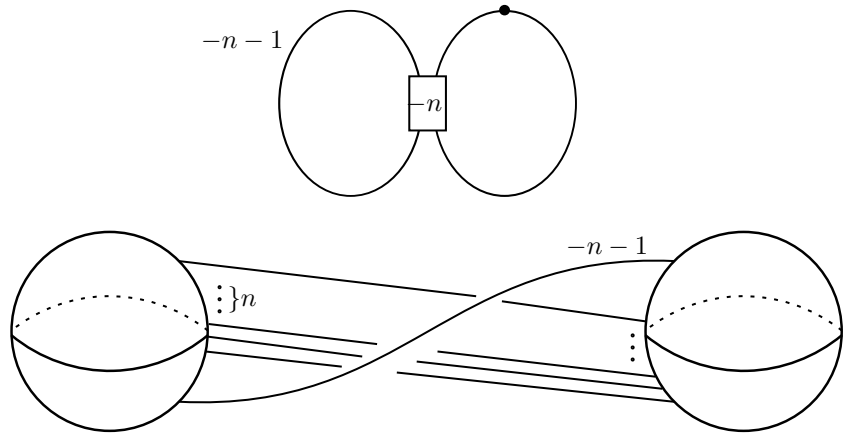}
\caption{{\bf Another Kirby diagram of $B_n$}}
\label{f:bnsphm}
\end{figure}

Having this Kirby diagram for $B_n$, we are now ready to put it in Legendrian standard form, as seen in Figure~\ref{f:bnsphstein}. (Note, this is the same Stein structure on $B_n$ as it recently appeared in \cite{LM}, for $q=1$.)  The orientation was chosen arbitrarily, but will remain fixed throughout. Observe, that the Legendrian knot $K_2^n$ in the diagram has the following classical invariants:
\begin{equation*}
tb(K_2^n) = w(K_2^n) - \lambda(K_2^n) = -(n-1) - 1 = -n 
\end{equation*}
\begin{equation*}
rot(K_2^n) = \lambda_- - \rho_+ = 1 \, . 
\end{equation*}

\begin{figure}[ht!]
\labellist
\small\hair 2pt
\pinlabel $-1$ at 425 133
\pinlabel $\}n$ at 185 100
\pinlabel $K_2^n$ at 180 16
\endlabellist
\centering
\includegraphics[scale=0.55]{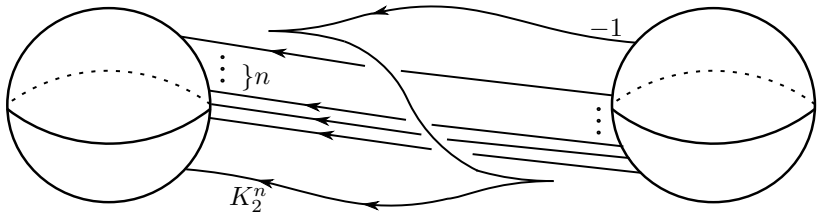}
\caption{{\bf Kirby diagram of $B_n$ with Stein structure $J_n$}}
\label{f:bnsphstein}
\end{figure}

\noindent Therefore, the framing with which the 2-handle is attached is precisely as dictated by Theorem~\ref{thm:g2}, namely $tb(K_2^n) - 1 = -n-1$.

Recall, that since the set of Stein structures of a 4-manifold is a subset of the set of almost-complex structures of a 4-manifold, then from the Stein surface $(B_n,J_n)$ we naturally get a symplectic 4-manifold $(B_n, \omega_n)$, where the symplectic form $\omega_n$ is induced by the almost-complex structure $J_n$.

Second, we present a symplectic structure $\omega'_n$ on $C_n$, also obtained from the Stein structure $J'_n$ on $C_n$, which we exhibit explicitly with a Legendrian link diagram (with no 1-handles). We label the unknots in the plumbing diagram of $C_n$, (as seen in Figure~\ref{f:cn2}), $W_1, W_2, \ldots W_{n-1}$. We put a Stein structure $J'_n$ on $C_n$, seen in Figure~\ref{f:cnstein}, by making the unknots, representing the spheres in the plumbing configuration, Legendrian in such a way that the framing of each unknot corresponds to the required framing as dictated by Theorem~\ref{thm:g3}: $tb(W_i) - 1$. Observe, that in this particular choice of Legendrian representatives of unknots, we have $rot(W_1) = -n$, $rot(W_2) = \cdots = rot(W_{n-1}) = 0$.

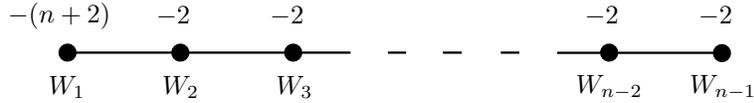
\begin{figure}[ht!]
\labellist
\small\hair 2pt
\pinlabel $-(n+2)$ at 30 4.7
\pinlabel $-2$ at 60 4.7
\pinlabel $-2$ at 90 4.7
\pinlabel $-2$ at 173 4.7
\pinlabel $-2$ at 203 4.7
\pinlabel $W_1$ at 32 2.5 
\pinlabel $W_2$ at 62 2.5 
\pinlabel $W_3$ at 92 2.5 
\pinlabel $W_{n-2}$ at 175 2.5 
\pinlabel $W_{n-1}$ at 205 2.5 
\endlabellist
\centering
\includegraphics[height=30mm, width=120mm]{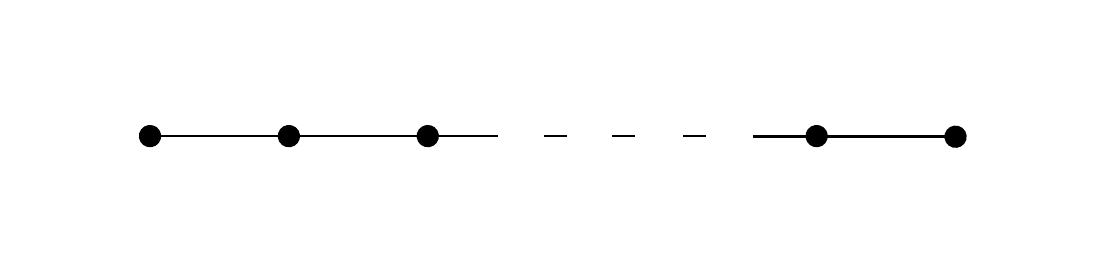}
\caption{{\bf Plumbing diagram of $C_n$, $n\geq2$}}
\label{f:cn2}
\end{figure}

\begin{figure}[ht!]
\labellist
\small\hair 2pt
\pinlabel $-1$ at 400 95
\pinlabel $-1$ at 540 95
\pinlabel $-1$ at 740 95
\pinlabel $-1$ at 860 95
\pinlabel $-1$ at 100 150
\pinlabel $W_1$ at 230 -10 
\pinlabel $W_2$ at 400 10 
\pinlabel $W_3$ at 540 10
\pinlabel $W_{n-2}$ at 720 10 
\pinlabel $W_{n-1}$ at 860 10 
\pinlabel $\leftarrow(n+1)\,\mathrm{cusps}$ at 480 175
\endlabellist
\centering
\includegraphics[height=33mm, width=125mm]{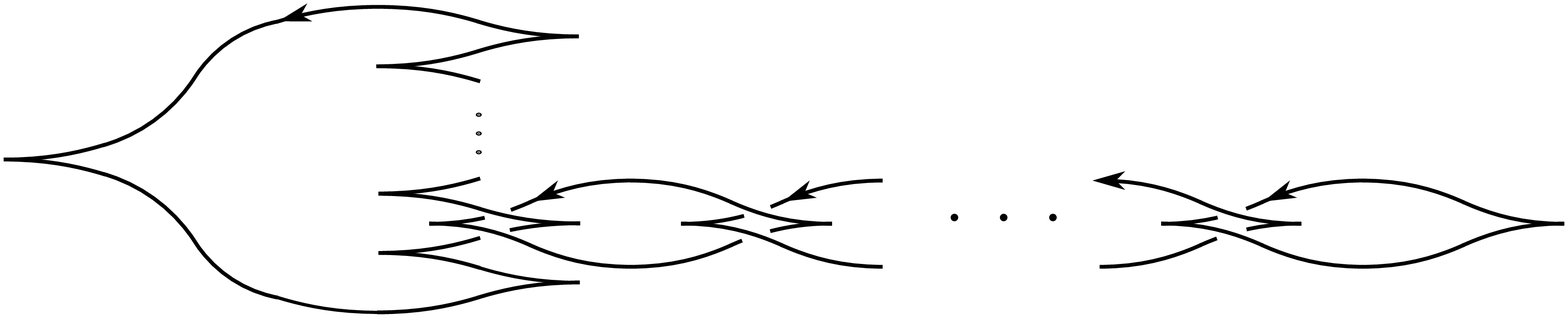}
\caption{{\bf Stein structure $J'_n$ on $C_n$}}
\label{f:cnstein}
\end{figure}

\section{Symplectic rational blow-up - main theorem}
\label{sec:srbumain}
\subsection{Lagrangian cores $\L_{n,q}$.}
\label{sec:srbuintro}
In this section we define the symplectic rational blow-up operation and prove the main theorem. It is important to note, that just like the symplectic blow-up is not unique because of the choice of radius of the removed 4-ball, so to, the symplectic rational blow-up will also not be unique due to the choice of the symplectic volume of the removed rational homology ball $B_n$. Moreover, we also have to make a choice of the symplectic structure on the $B_n$s. Therefore, we will go further, and show that the existence of a certain 2-dimensional Lagrangian core (see section~\ref{sec:bn}) in a symplectic manifold $(X,\omega)$ will have a standard neighborhood that will be our desired symplectic rational homology ball $(B_n,\omega_n)$ as in section~\ref{sec:sympbncn}.

Now we will describe the construction of our Lagrangian cores. First, we take an embedding $\gamma:S^1 \hookrightarrow (X,\omega)$. Next, we consider a Lagrangian immersion $\L:D^2 \looparrowright (X,\omega)$, (an embedding on the interior of $D$), such that its boundary ``wraps around" $\gamma(S^1)$, with winding number $n$, so $\gamma(S^1) \hookrightarrow \L(\partial D)$. There is another winding number $q$ that comes in to the picture, so we are going to call this Lagrangian disk immersion $\L_{n,q}$. Let $P$ be the following bundle over $\gamma(S^1)$:
\begin{equation*}
P = \bigcup_{z \in \gamma(S^1)} \{ \text{plane} \,\, \pi | \pi \subset T_z X,\, \text{oriented},\, \omega(\pi) = 0,\, T_z(\gamma (S^1)) \subset \pi \} \, . 
\end{equation*}
%\vspace{.1in}

\noindent Because we are restricting to those planes $\pi$ that contain $T_z(\gamma (S^1))$, the bundle $P$ is an $S^1$-bundle. So, after a choice of trivialization,  we have $P \cong S^1 \times \gamma(S^1)$, and a map:
\begin{eqnarray}
\label{eq:lnqhat}
\widehat{\L_{n,q}}&:& \partial D \rightarrow P \cong S^1 \times \gamma(S^1)  \\
\widehat{\L_{n,q}}&:& x \longmapsto (\L_{n,q})_{\ast}(T_x D) \nonumber
\end{eqnarray}
%\vspace{.1in}

\noindent where $n$ is the degree of the map $\widehat{\L_{n,q}}$ on the first component, and $q$ on the second. Note, that before a choice of trivialization of $P$, $q$ is only defined mod $n$.

Now we state the formal definition of the Lagrangian ``cores", $\L_{n,q}$:

\begin{defn}
\label{d:lnq}
Let $\L_{n,q}: D \looparrowright (X,\omega)$ be a smooth Lagrangian immersion of a 2-disk $D$ into a symplectic 4-manifold $(X,\omega)$, with $n \geq 2$ an integer, and $q$ is an integer defined mod $n$, assuming the following conditions:
\begin{enumerate}[(i)]
\item
$\L_{n,q}(D - \partial D) \hookrightarrow (X,\omega)$ is a smooth embedding.
\item 
There exists a smooth embedding $\gamma:S^1 \hookrightarrow (X,\omega)$ such that $\gamma(S^1) \hookrightarrow \L_{n,q}(\partial D)$.
\item
The pair $(n,q)$ are defined to be the degrees of the maps on the first and second component, respectively of the map $\widehat{\L_{n,q}}: \partial D \rightarrow P \cong S^1 \times \gamma(S^1)$ as defined in (\ref{eq:lnqhat}).

\item 
The map $\widehat{\L_{n,q}}$ is injective, so for any points $x,y \in \partial D$ if $\L_{n,q}(x) = \L_{n,q}(y)$ then $(\L_{n,q})_{\ast}(T_x(D)) \neq (\L_{n,q})_{\ast}(T_y(D))$.

\end{enumerate}
\end{defn}

\noindent Figure~\ref{f:L3} is an illustration of how $\L_{n,q}(D)$ looks like near $\gamma(S^1)$, for $n=3$ and $q=1$. Note, we will use $\L_{n,q}$ to also denote its image in $(X,\omega)$.

\subsection{Statement of the main theorem.}
\label{sec:srbuthm}
Now we are ready to state the main theorem:

\begin{thm}{\textbf{Symplectic Rational Blow-Up.}}
\label{thm:srbu}
Suppose $\L_{n,1} \subset (X,\omega)$, is as in Definition~\ref{d:lnq} with $q=1$, then for some small $\lambda > 0$, there  exists a symplectic embedding of $(B_n,\lambda \omega_n)$ in $(X,\omega)$, and for some $\lambda_0 < \lambda$ and $\mu > 0$, there exists a symplectic 4-manifold $(X',\omega')$ such that $(X',\omega') = ((X,\omega) - (B_n,\lambda_0 \omega_n)) \cup_{\phi} (C_n,\mu \omega'_n)$, where $\phi$ is a symplectic map, and $(B_n,\omega_n)$ and $(C_n,\omega'_n)$ are the symplectic manifolds as defined in section~\ref{sec:sympbncn}. $(X',\omega')$ is called the \textbf{symplectic rational blow-up} of $(X,\omega)$.
\end{thm}

\begin{proof}

The proof of the theorem will follow from Lemmas ~\ref{l:srbu1} and ~\ref{l:srbu2} below, but first we will introduce some notation.

We express $\L_{n,q} \subset (X,\omega)$ as a union: 
\begin{equation}
\label{eq:lnqsplit}
\L_{n,q} = \Sigma_{n,q} \cup \Delta ,
\end{equation}

\noindent where $\Sigma_{n,q}$ is the image of a collar neighborhood of $\partial D  \subset D$, $C_D$, and $\Delta$ is the image of the remainder $D-C_D$. First, we will present a model of $\Sigma_{n,q}$ explicitly by expressing it in terms of local coordinates. 

For $\L_{n,q}$, the respective $\gamma(S^1) \hookrightarrow (X,\omega)$, as in Definition~\ref{d:lnq}, will have a neighborhood, $S^1 \times D^3$ with standard Darboux coordinates: $(\theta,x,u,v)$ with the symplectic form $\omega = d\theta \wedge dx + du \wedge dv$, where $\theta$ is a $2\pi$-periodic coordinate on $S^1$, and $x,u,v$ are the standard coordinates on $D^3$. Parameterizing $C_D$ by $(t,s)$ with $0 \leq t < 2\pi$ and $0 \leq s \leq \epsilon$ for some small $\epsilon$, Definition~\ref{d:lnq} implies that without loss of generality, $\Sigma_{n,q}(t,s)$ can be expressed as:
\begin{equation}
\label{eq:sigmanqts}
\Sigma_{n,q}(t,s) = (nt,x(t,s),s\cos(\psi_q(t,s)),-s\sin(\psi_q(t,s)))
\end{equation}
%\vspace{.1in} 

\noindent where $x(t,s)$ and $\psi_q(t,s)$ are smooth functions with $x(0,s) = x(2\pi,s)$ and $\psi_q(2\pi,s) - \psi_q(0,s) = q(2\pi)$. We observe that at $s=0$ we have:
\begin{equation*}
\Sigma_{n,q}(t,0) = (nt,0,0,0) = \gamma(S^1)  \, . 
\end{equation*}
%\vspace{.1in} 

\noindent Thus, the numbers in the pair $(n,q)$ as they appear in (\ref{eq:sigmanqts}), are the degrees of the maps in part (iii) of Definition~\ref{d:lnq}.

Next, we switch to  somewhat more convenient coordinates $(\theta,x,\tau,\rho)$, (sometimes referred to as action-angle coordinates) where:
\begin{equation*}
\theta \rightarrow \theta, \hspace{.1in} x \rightarrow x, \hspace{.1in} u \rightarrow \sqrt{2\rho}\cos \tau, \hspace{.1in} v \rightarrow -\sqrt{2\rho}\sin \tau \, .
\end{equation*}
%\vspace{.1in} 

\noindent This coordinate change is symplectic, since the symplectic form remains the same: $\omega = d\theta \wedge dx + d\tau \wedge d\rho$. We can reparameterize $\Sigma_{n,q}$ with $(t,I)$, $0 \leq t < 2\pi$ and $0 \leq I \leq \epsilon '$, where $I = \frac{1}{2}s^2$, and so (\ref{eq:sigmanqts})  in $(\theta,x,\tau,\rho)$ coordinates becomes:
\begin{equation*}
\Sigma_{n,q}(t,I) = (nt,x(t,I),\psi_q(t,I),I) \, . 
\end{equation*}
%\vspace{.1in}

The Lagrangian condition $\omega_{|T_{\L_{n,q}(D)}X}=0$ imposes further restrictions on $x(t,I)$, thus $\Sigma_{n,q}(t,I)$ can be given as follows:
\begin{equation}
\label{eq:sigmanqti}
\Sigma_{n,q}(t,I) = (nt,-\frac{q}{n}I\frac{\partial \psi_q}{\partial t} + \int \frac{q}{n}I\frac{\partial^2 \psi_q}{\partial I \partial t}\,dI,\psi_q(t,I),I)  \, .
\end{equation}
%\vspace{.1in}

\noindent A particular example is when $\psi_q(t,I) = qt$, this will be called $\Sigma^{\sharp}_{n,q}$:
\begin{equation}
\label{eq:sigmastd}
\Sigma^{\sharp}_{n,q}(t,I) = (nt,-\frac{q}{n}I,qt,I) \, .
\end{equation}
%\vspace{.1in}

\noindent Again, we refer the reader to Figure~\ref{f:L3} for an illustration of $\Sigma^{\sharp}_{n,q}$ for $n=3$ and $q=1$.

\begin{lemma}
\label{l:srbu1}
Let $\L_{n,q} \subset (X,\omega)$ be as in Definition~\ref{d:lnq}. Then there exists another $\L_{n,q}^{\sharp} \subset (X,\omega)$, also as in Definition~\ref{d:lnq}, such that if $\L_{n,q} = \Sigma_{n,q} \cup \Delta$, (as defined in (\ref{eq:lnqsplit})), then $\L_{n,q}^{\sharp} = \Sigma_{n,q}^{\sharp} \cup \Delta^{\sharp}$, where $\Sigma_{n,q}^{\sharp}$ is as in (\ref{eq:sigmastd}) and $\Delta^{\sharp}$ agrees with $\Delta$ everywhere except for a small neighborhood of its boundary. We will refer to such $\L_{n,q}^{\sharp}$s as the ``good" ones. Thus, all the ``good" $\L_{n,q}$s are the ones which are standard in a neighborhood of $\gamma(S^1)$.
\end{lemma}

\begin{lemma}
\label{l:srbu2}
Let $\L_{n,q}^{\sharp}$ and $\check{\L}_{n,q}^{\sharp}$ be both ``good" $\L_{n,q}$s, in accordance with Definition~\ref{d:lnq} and Lemma~\ref{l:srbu1}, then they will have symplectomorphic neighborhoods in $(X,\omega)$.
\end{lemma}

Note, the above Lemmas are meant to mirror the standard Weinstein Lagrangian embedding theorem. First, we will prove Lemma~\ref{l:srbu1} by constructing a Hamiltonian vector flow that will take $\Sigma_{n,q}$ to $\Sigma^{\sharp}_{n,q}$. Second, we will prove Lemma~\ref{l:srbu2} using Lemma~\ref{l:srbu1} and a relative Moser type argument.

\begin{proof}{Proof of Lemma~\ref{l:srbu1}.}
We construct a Hamiltonian $H$ with flow 
\begin{equation*}
\varphi_{\alpha}:nbhd(\widetilde{\gamma(S^1)}) \rightarrow nbhd(\widetilde{\gamma(S^1)}), 
\end{equation*}

\noindent for $0 \leq \alpha \leq 1$, where $\widetilde{\gamma(S^1)}$ is the $n$-sheeted covering space of $\gamma(S^1)$. Note, we choose $\epsilon '$ small enough such that $\widetilde{\Sigma^{\sharp}_{n,q}(t,I)} \subset nbhd(\widetilde{\gamma(S^1)})$.  $H$ and $\varphi_\alpha$ are as given in (\ref{eq:flow}) and (\ref{eq:ham}) below on $\widetilde{\Sigma^{\sharp}_{n,q}}(t,I)$ and are $0$ otherwise:
\small
\begin{equation}
\label{eq:flow}
\varphi_{\alpha}(\theta,x,\tau,\rho) = (\theta,x -(\frac{\partial f}{\partial \theta}\rho - \int \frac{\partial^2 f}{\partial \rho \partial \theta}\rho\,d\rho)\alpha,\tau + f(\theta,\rho)\alpha,\rho)
\end{equation}
\normalsize
\begin{equation}
\label{eq:ham}
H(\theta,x,\tau,\rho) = \int f(\theta,\rho) \,d\rho 
\end{equation}
%\vspace{.1in}

\noindent for some continuous function $f$.

The following calculation shows that $\varphi_\alpha$ preserves the symplectic form $\omega = d\theta \wedge dx + d\tau \wedge d\rho$, and that it is indeed the Hamiltonian flow for the $H$ above.
\small
\begin{eqnarray*}
& & d\theta \wedge d(x -(\frac{\partial f}{\partial \theta}\rho - \int \frac{\partial^2 f}{\partial \rho \partial \theta}\rho\,d\rho)\alpha) + d(\tau + f(\theta,\rho)\alpha) \wedge d\rho   \\ %%%%%%%%%%%%%%%%%%%%%%%%%%%%%%%%%%%%%%%%%%%%%%%%%%%%%%%%%%%%%%%%%%%%%%%%%%%%%%%%%%%%%%%%%%%%%%%% 
&=& d\theta \wedge (dx - \alpha (\frac{\partial^2 f}{\partial \theta^2}\rho d\theta +  \frac{\partial f}{\partial \theta} d\rho + \frac{\partial^2 f}{\partial \rho \partial \theta} \rho d\rho - \frac{\partial}{\partial \theta}(\int \frac{\partial^2 f}{\partial \rho \partial \theta}\rho\,d\rho)  d\theta - \frac{\partial^2 f}{\partial \rho \partial \theta} \rho d\rho))   \\ %%%%%%%%%%%%%%%%%%%%%%%%%%%%%%%%%%%%%%%%%%%%%%%%%%%%%%%%%%%%%%%%%%%%%%%%%%%%%%%%%%%%%%%%%%%%%%%%
&+& (d\tau + \alpha(\frac{\partial f}{\partial \theta} d\theta + \frac{f}{\partial \rho} d\rho)) \wedge d\rho   \\ %%%%%%%%%%%%%%%%%%%%%%%%%%%%%%%%%%%%%%%%%%%%%%%%%%%%%%%%%%%%%%%%%%%%%%%%%%%%%%%%%%%%%%%%%%%%%%%%
&=& d\theta \wedge dx - \alpha \frac{\partial f}{\partial \theta} d\theta \wedge d\rho + d\tau \wedge d\rho + \alpha \frac{\partial f}{\partial \theta} d\theta \wedge d\rho   \\ %%%%%%%%%%%%%%%%%%%%%%%%%%%%%%%%%%%%%%%%%%%%%%%%%%%%%%%%%%%%%%%%%%%%%%%%%%%%%%%%%%%%%%%%%%%%%%%%
&=& d\theta \wedge dx + d\tau \wedge d\rho \, .  
\end{eqnarray*}
\normalsize
%\vspace{.1in}

\noindent Also,
\small
\begin{equation*}
\frac{d}{d\alpha}\varphi_\alpha = (0, - \frac{\partial f}{\partial \theta}\rho + \int \frac{\partial^2 f}{\partial \rho \partial \theta}\rho\,d\rho, f(\theta,\rho), 0) = (\frac{\partial H}{\partial x},- \frac{\partial H}{\partial \theta},\frac{\partial H}{\partial \rho},-\frac{\partial H}{\partial \tau}) \, . 
\end{equation*}
\normalsize
%\vspace{.1in}

If we let $p_n : nbhd(\widetilde{\gamma(S^1)}) \rightarrow nbhd(\gamma(S^1))$ be the $(n:1)$ covering map, then we have $p_n \circ \varphi_1(\widetilde{\Sigma^{\sharp}_{n,q}}) = \Sigma_{n,q}$, taking $f(nt,I) = \psi_q(t,I) - qt$, as seen in the equation below:
\small
\begin{eqnarray*}
p_n \circ \varphi_1(\widetilde{\Sigma^{\sharp}_{n,q}})(t,I) &=& (nt,-\frac{q}{n}I -\frac{\partial f(nt,I)}{\partial(nt)}I + \int \frac{\partial^2 f(nt,I)}{\partial I \partial (nt)}I\,dI,t + f(nt,I),I)  \\ 
&=& (nt,-\frac{q}{n}I\frac{\partial \psi_q}{\partial t} + \int \frac{q}{n}I\frac{\partial^2 \psi_q}{\partial I \partial t}\,dI,\psi_q(t,I),I)  \\
&=& \Sigma_{n,q}(t,I) \, . 
\end{eqnarray*}
\normalsize
%\vspace{.1in}

Note, in order for $p_n \circ \varphi_\alpha (\widetilde{\Sigma^{\sharp}_{n,q}})$ to remain being a ``$\Sigma_{n,q}$" for all $0 \leq \alpha \leq 1$, (and not ``tear" as $\alpha$ goes from $0$ to $1$), we must have 
\begin{equation*}
\left[q(2\pi) +(\psi_q(2\pi,I) - q(2\pi)\alpha\right]- [q(0) + (\psi_q(I,0) - q(0))\alpha] 
\end{equation*}
%\vspace{.1in}

\noindent be an integer multiple of $2\pi$ for all $0 \leq \alpha \leq 1$. This implies:
\begin{equation*}
\psi_q(2\pi,I) - \psi_q(0,I) = q(2\pi) \, . 
\end{equation*} 
%\vspace{.1in}

\noindent Which is precisely the condition that $\psi_q(t,I)$ needs to have in the definition of $\Sigma_{n,q}(t,I)$. Hence, whenever we have  $\L_{n,q} \subset (X,\omega)$, we can always find a ``good" $\L_{n,q}^{\sharp} \subset (X,\omega)$, which looks ``standard" near $\gamma(S^1)$, by the map $p_n \circ \varphi_1^{-1}(\widetilde{\Sigma_{n,q}}) = \Sigma^{\sharp}_{n,q}$, with $\L_{n,q}^{\sharp} = \Sigma_{n,q}^{\sharp} \cup \Delta^{\sharp}$. (We have $\Delta^{\sharp}$, since the map
$p_n \circ \varphi_1^{-1}$ gets smoothed off near $\partial \Delta$.) \end{proof}

\begin{proof}{Proof of Lemma~\ref{l:srbu2}.} In order to prove this lemma, we will be using the relative Moser's theorem, stated below:

\begin{lemma}{\textbf{Relative Moser's Theorem.}}
\label{l:moser}
\cite{EliM} Let $\omega_t$ be a family of symplectic forms on a compact manifold $W$ with full-dimensional submanifold $W_1$, such that $\omega_t = \omega_0$ over an open neighborhood of $W_1$ and the relative cohomology class $\left[\omega_t - \omega_0\right] \in H^2(W,W_1)$ vanishes for all $t \in \left[0,1\right]$. Then there exists an isotopy $\Phi_t:W \rightarrow W$ which is fixed on an open neighborhood of $W_1$ and such that $\Phi_t^*(\omega_0) = \omega_t$, $t \in \left[0,1\right]$.
\end{lemma}

(Note, in \cite{EliM} this thereom is stated for the pair $(W,\partial W)$, however, the proof directly extends to the pair $(W,W_1)$.)

Let $\L_{n,q}^{\sharp}$ be a ``good" $\L_{n,q}$ immersed disk, and let $\L_{n,q}^{0,\sharp} \hookrightarrow (X_0,\omega_0)$ be some particular ``good" $\L_{n,q}$ immersed disk in a symplectic 4-manifold $(X_0,\omega_0)$. Let $\Sigma^{\sharp, \delta}_{n,q}(t,I) \subset \Sigma^{\sharp}_{n,q}(t,I)$ be such that $0 \leq t < 2\pi$ and $\delta \leq I < \epsilon '$. Then, we let
\begin{eqnarray*}
\L_{n,q}^{\sharp, \delta} &=& \Sigma^{\sharp, \delta}_{n,q}(t,I) \cup \Delta^{\sharp}  \\
\L_{n,q}^{0, \sharp, \delta} &=& \Sigma^{\sharp, \delta}_{n,q}(t,I) \cup \stackrel{\circ}{\Delta}^{\sharp} \, . 
\end{eqnarray*}

\begin{figure}[ht!]
\labellist
\small\hair 2pt
\pinlabel $\O_{\L_{n,q}^{0,\sharp, \delta}}$ at 220 113
\pinlabel $\L_{n,q}^{0,\sharp, \delta}$ at 300 87
\pinlabel $\O_{\Sigma_{n,q}^{\sharp, \delta}}$ at 67 87
\pinlabel $\Sigma_{n,q}^{\sharp, \delta}$ at 110 87
\pinlabel $nbhd(\gamma(S^1))$ at 85 35
\endlabellist
\centering
\includegraphics[height=70mm, width=135mm]{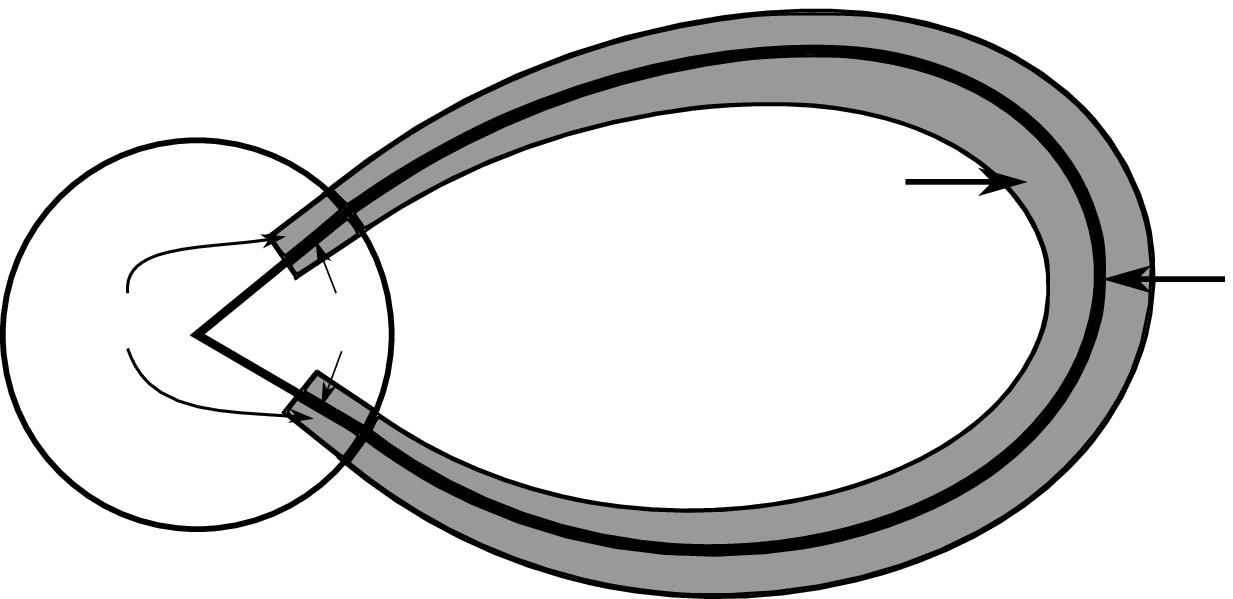}
\caption{\bf Schematic diagram of $\O_{\L_{n,q}^{0,\sharp, \delta}}$}
\label{f:lcore}
\end{figure}

\noindent Also, let $\nu(X,\L_{n,q}^{\sharp, \delta})$ and $\nu(X_0,\L_{n,q}^{0, \sharp, \delta})$ be normal bundles of $\L_{n,q}^{\sharp, \delta}$ and $\L_{n,q}^{0, \sharp, \delta}$ respectively. We also denote 
\begin{eqnarray*}
N_{\Sigma_{n,q}^{\sharp, \delta}} \subset N_{\L_{n,q}^{\sharp, \delta}} &\subset& \nu(X,\L_{n,q}^{\sharp, \delta})  \\
\O_{\Sigma_{n,q}^{\sharp, \delta}} \subset \O_{\L_{n,q}^{0,\sharp, \delta}} &\subset& \nu(X_0,\L_{n,q}^{0,\sharp, \delta}) 
\end{eqnarray*}

\noindent to be the neighborhoods of $\Sigma_{n,q}^{\sharp, \delta}$, $\L_{n,q}^{\sharp, \delta}$ and $\L_{n,q}^{0,\sharp, \delta}$ in their respective normal bundles. Refer to Figure~\ref{f:lcore} for a schematic diagram. We construct a bundle map:
\begin{equation*}
B_0: T_x(\nu(X_0,\L_{n,q}^{0,\sharp, \delta})) \longrightarrow T_y(\nu(X,\L_{n,q}^{\sharp, \delta}))  
\end{equation*}
%\vspace{.1in}

\noindent for $x \in \L_{n,q}^{0,\sharp, \delta}$ and $y \in \L_{n,q}^{\sharp, \delta}$ such that ${B_0}|_{\Sigma_{n,q}^{\sharp, \delta}} = Id$. By  the Whitney Extension theorem \cite{Wh}, we have a map
\begin{equation*} 
\phi_0: \O_{\L_{n,q}^{0,\sharp, \delta}} \rightarrow N_{\L_{n,q}^{\sharp, \delta}}  
\end{equation*}
 
\noindent with $\phi = B_0$ on $T_{\L_{n,q}^{0,\sharp, \delta}}(\nu(X_0,\L_{n,q}^{0,\sharp, \delta}))$ and $\phi_0^\ast(\omega) = \omega_0$ on $\O_{\Sigma_{n,q}^{\sharp, \delta}}$.

Next, we define a family of symplectic forms:
\begin{equation*}
\omega_t = (1-t)\omega_0 + t\phi_0^{\ast}(\omega) \,\, \text{for} \,\, t \in [0,1] \, . 
\end{equation*}
%\vspace{.1in}

\noindent We get $\omega_t - \omega_0 = t(\phi_0^{\ast}(\omega) - \omega_0) = 0$, for all $t \in [0,1]$ on some open neighborhood of $\O_{\Sigma_{n,q}^{\sharp, \delta}}$. We can do this by making our $\epsilon '$ a bit smaller. Moreover, we can pass down to the relative homology class:
\begin{equation*}
[\omega_t - \omega_0] \equiv [t(\phi_o^{\ast}(\omega) - \omega_0)] \in H^2(\O_{\L_{n,q}^{0,\sharp, \delta}},\O_{\Sigma_{n,q}^{\sharp, \delta}}) \, .  
\end{equation*}
%\vspace{.1in}

\noindent This relative class $[\omega_t - \omega_0]$ will vanish since $\phi_0^\ast(\omega) = \omega_0$ on $\O_{\Sigma_{n,q}^{\sharp, \delta}}$. Thus, we can use relative Moser's theorem (Lemma~\ref{l:moser}), with $W=\O_{\L_{n,q}^{0,\sharp, \delta}}$ and $W_1 = \O_{\Sigma_{n,q}^{\sharp, \delta}}$, and we get an isotopy $\Phi_t : \O_{\L_{n,q}^{0,\sharp, \delta}} \rightarrow \O_{\L_{n,q}^{0,\sharp, \delta}}$ such that $\Phi_1^{\ast}(\omega_0) = \omega_1 = \phi_0^\ast(\omega)$. We define the map $\Phi_{\sharp} = \phi_0 \circ \Phi_1^{-1}$, and obtain:
\begin{equation*}
\Phi_{\sharp}: \O_{\L_{n,q}^{0,\sharp, \delta}} \rightarrow N_{\L_{n,q}^{\sharp, \delta}} \,\,\, \text{with} \,\,\, \Phi_{\sharp}^{\ast}(\omega) = \omega_0 \, .  
\end{equation*}
%\vspace{.1in}

Likewise, we can obtain a symplectomorphism $\check{\Phi}_{\sharp}: N_{\check{\L}_{n,q}^{\sharp, \delta}} \rightarrow \O_{\L_{n,q}^{0,\sharp, \delta}}$. By composing $\Phi_{\sharp}$ and $\check{\Phi}_{\sharp}$, we get a symplectomorphism:
\begin{equation*}
\Phi : N_{\check{\L}_{n,q}^{\sharp, \delta}} \rightarrow N_{\L_{n,q}^{\sharp, \delta}}  \, ,
\end{equation*}
%\vspace{.1in}

\noindent which extends to map between $\check{\L}_{n,q}^{\sharp}$ and $\L_{n,q}^{\sharp}$, since they are both ``good" immersed disks, and are the same on $\Sigma_{n,q}^{\sharp}$.

Now to complete the proof of Lemma~\ref{l:srbu2}, we will construct a particular model of a neighborhood of such an immersed Lagrangian disk $\L_{n,q}^{0, \sharp} = \Sigma^{\sharp}_{n,q}(t,I) \cup \stackrel{\circ}{\Delta}$. We will do this by symplectically gluing $N_{\Sigma_{n,q}^{\sharp}}$ to $N_B \subset T^{\ast}(B)$, where $T^{\ast}(B)$ is just the cotangent space of a 2-disk $B$, and $N_B$ is its neighborhood in $T^{\ast}(B)$. With the identification of $\Sigma_{n,q}^{\sharp,\delta}$ with $C_B$, a collar neighborhood of the boundary of disk $B$, we can construct a symplectomorphism $\Psi$ between $N_{\Sigma_{n,q}^{\sharp,\delta}} \subset \nu(X,\L_{n,q}^{\sharp, \delta})$ and $N_{C_B} \subset T^{\ast}(B)$, by a similar Moser type argument as used above. We then symplectically glue $N_{\Sigma_{n,q}^{\sharp}}$ to $N_B$ via $\Psi$. \end{proof}

\subsection{Showing $(nbhd\,\L_{n,1}^{\sharp}) \cong (B_n,\omega_n)$.}
Now that we have shown that a neighborhood of a ``good" Lagrangian core $\L_{n,1}^{\sharp}$ is entirely standard, we will now show that this standard neighborhood is in fact equivalent to $(B_n,\omega_n)$ for each $n \geq 2$,  where $\omega_n$ are the symplectic forms induced on the rational homology balls $B_n$ by the Stein structures $J_n$, in section~\ref{sec:sympbncn}. Note, there is a choice in the size of a neighborhood of $\L_{n,1}^{\sharp}$ which corresponds to the choice of the symplectic volume of the rational homology ball $B_n$; this is the source of the non-uniqueness of the symplectic rational blow-up operation, as mentioned in section~\ref{sec:srbuintro}.

\begin{lemma}
\label{l:srbu3}
There exists a neighborhood of $\L_{n,1}^{\sharp}$ in $(X,\omega)$, $N({\L_{n,1}^{\sharp}})$, such that there exists a symplectomorphism 
\begin{equation}
\label{eq:sympplus}
f: (N(\L_{n,1}^{\sharp}), \omega |_{N(\L_{n,1}^{\sharp})})^+ \rightarrow  (B_n,\omega_n)^+
\end{equation}

\noindent where $(N(\L_{n,1}^{\sharp}), \omega |_{N(\L_{n,1}^{\sharp})})^+$ and $(B_n,\omega_n)^+$ are the symplectic completions (see for example \cite{OzSt}) of $(N(\L_{n,1}^{\sharp}),$ $\omega |_{N(\L_{n,1}^{\sharp})})$ and $(B_n,\omega_n)$ respectively. 
\end{lemma}

\begin{proof}

Recall that the ``good" Lagrangian cores $\L_{n,1}^{\sharp}$ can be expressed as a union $\L_{n,1}^{\sharp} = \Sigma_{n,1}^{\sharp}(t,I) \cup \Delta^{\sharp}$, and that $\Sigma^{\sharp, \delta}_{n,q}(t,I) \subset \Sigma^{\sharp}_{n,q}(t,I)$ is such that $0 \leq t < 2\pi$ and $\delta \leq I < \epsilon '$. We fix a number $0 < a < \epsilon'$ and let:
\begin{equation}
\label{eq:knotkn1}
\partial (\Sigma_{n,1}^{\sharp} - \Sigma_{n,1}^{\sharp, a}) = \K_{n,1}  
\end{equation}
%\vspace{.1in}

\noindent where $\K_{n,1}$ is a knot in $\partial(S^1 \times D^3) \cong S^1 \times S^2$, and the spheres $S^2$ have radius $a$. The knot $\K_{n,1}$ can be described with respect to the $(\theta, x, \tau, \rho)$ coordinates, introduced in section~\ref{sec:srbuthm}, as follows:
\begin{equation}
\K_{n,1}(t) = (nt, -\frac{a}{n}, t, a)  \, . 
\end{equation}
%\vspace{.1in}

We observe that $\K_{n,1}$ is a Legendrian knot with respect to the standard (tight) contact structure on $S^1 \times S^2$, which has the contact 1-form
\begin{equation}
\alpha = -x d\theta - \rho d\tau
\end{equation}
%\vspace{.1in}

\noindent with the restriction to the spheres $x^2 + 2\rho = a^2$.

In light of Eliashberg's classification of Stein handlebodies \cite{Eliash} and Go\-mpf's Kirby-Legendrian moves \cite{Gompf1}, in order to show that a neighborhood of $\L_{n,1}^{\sharp}$ is the same symplectic manifold as $(B_n,\omega_n)$, then all we have to show is that $\K_{n,1}$ in (\ref{eq:knotkn1}) is the same Legendrian knot as $K_2^n$ in Figure~\ref{f:bnsphstein}. We will show this by presenting the knot $\K_{n,1}$ in $S^1\times S^2$ in an alternate way, and showing that this is equivalent to the presentation of the knot $K_2^n$ in \textit{standard form} as in Figure~\ref{f:bnsphstein}.

In (\cite{Gompf1}, section 2) Gompf presents an alternate way of presenting a knot in $S^1 \times S^2$, we recreate this method here. We want to pull back the contact 1-form $\alpha = -x d\theta - \rho d\tau$ to $\R^3$ using cylindrical coordinates $(\theta, r, \varpi)$, by stereographically projecting all of the spheres $S^2$, (with radius $a$), in $S^1 \times S^2$. Thus, when we perform the stereographic projections, we switch from coordinate system $(\theta, x, \tau, \rho)$ to $(\theta, r, \varpi)$, such that:
\begin{eqnarray*}
\theta & = & \theta  \\
     x & = & \frac{a(r^2 - 1)}{r^2 +1}  \\
  \tau & = & -\varpi  \\
  \rho & = & \frac{2a^2r^2}{(r+1)^2}   \, .
\end{eqnarray*}
%\vspace{.1in}

Consequently, the contact 1-form $\alpha = -x d\theta - \rho d\tau$ restricted to the spheres $x^2 + 2\rho = a^2$, becomes the following contact 1-form on $S^1 \times (S^2 - \left\{poles\right\})$:
\begin{equation*}
\tilde{\alpha} = d\varpi + \frac{1-r^4}{2ar^2}d\theta \, ,  
\end{equation*}
%\vspace{.1in}

\noindent which after rescaling pulls back to standard contact 1-form on $\R^3$,  
\begin{equation*}
\alpha_{std} = dZ + XdY  
\end{equation*}
%\vspace{.1in}

\noindent (with the $Z$ coordinate being $2\pi$-periodic). As a result, we can present knots in $S^1 \times S^2$ by their standard \textit{front} projections into the $Y$-$Z$ plane, i.e. by projecting them to the $\theta$-$\varpi$ ``plane" $\R^2 / 2\pi \Z^2$. Thus, one can alternately present knots in $S^1 \times S^2$ by disconnected arcs in a square, corresponding to $\R^2 / 2\pi \Z^2$.

Now we will present the knot $\K_{n,1}$, using this alternate presentation. First, we transfer the knot $\K_{n,1}$ into $(\theta, r, \varpi)$ coordinates,
\begin{equation}
\tilde{\K}_{n,1} = (nt, C_{a,n}, -t) \, ,
\end{equation}  
%\vspace{.1in}

\noindent where $C_{a,n}$ is just a constant depending on $a$ and $n$. Figure~\ref{f:Kn1} depicts the \textit{front} projection of $\tilde{\K}_{n,1}$ onto the $\theta$-$\varpi$ plane, (after we shift it in the $\theta$-coordinate, and take $-\frac{\pi}{2} \leq t \leq \frac{3\pi}{2}$). We then perform Gompf's move 6 (see \cite{Gompf1}, Figure 11), which in effect swings the knot around the 1-handle, and we obtain the knot as seen in Figure~\ref{f:Kn1cusp}, which is isotopic to the knot $K_2^n$ in \textit{standard form} in Figure~\ref{f:bnsphstein}.

\begin{figure}[ht]
\begin{minipage}[b]{0.45\linewidth}
\centering
\includegraphics[scale=0.5]{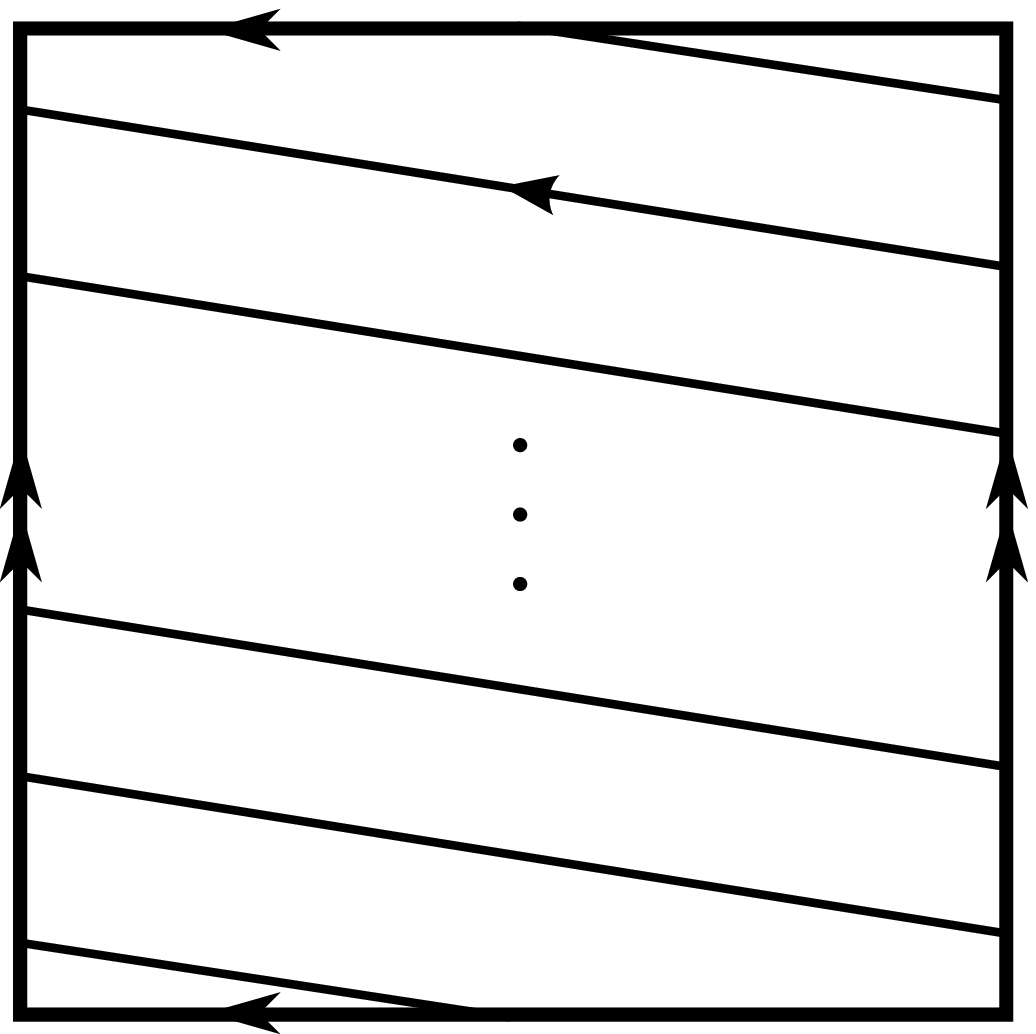}
\caption{ }
\labellist
\small\hair 2pt
\pinlabel $\}n$ at 180 150
\pinlabel $\}n$ at -230 150
\endlabellist
\label{f:Kn1}
\end{minipage}
\hspace{0.5cm}
\begin{minipage}[b]{0.45\linewidth}
\centering
\includegraphics[scale=0.545]{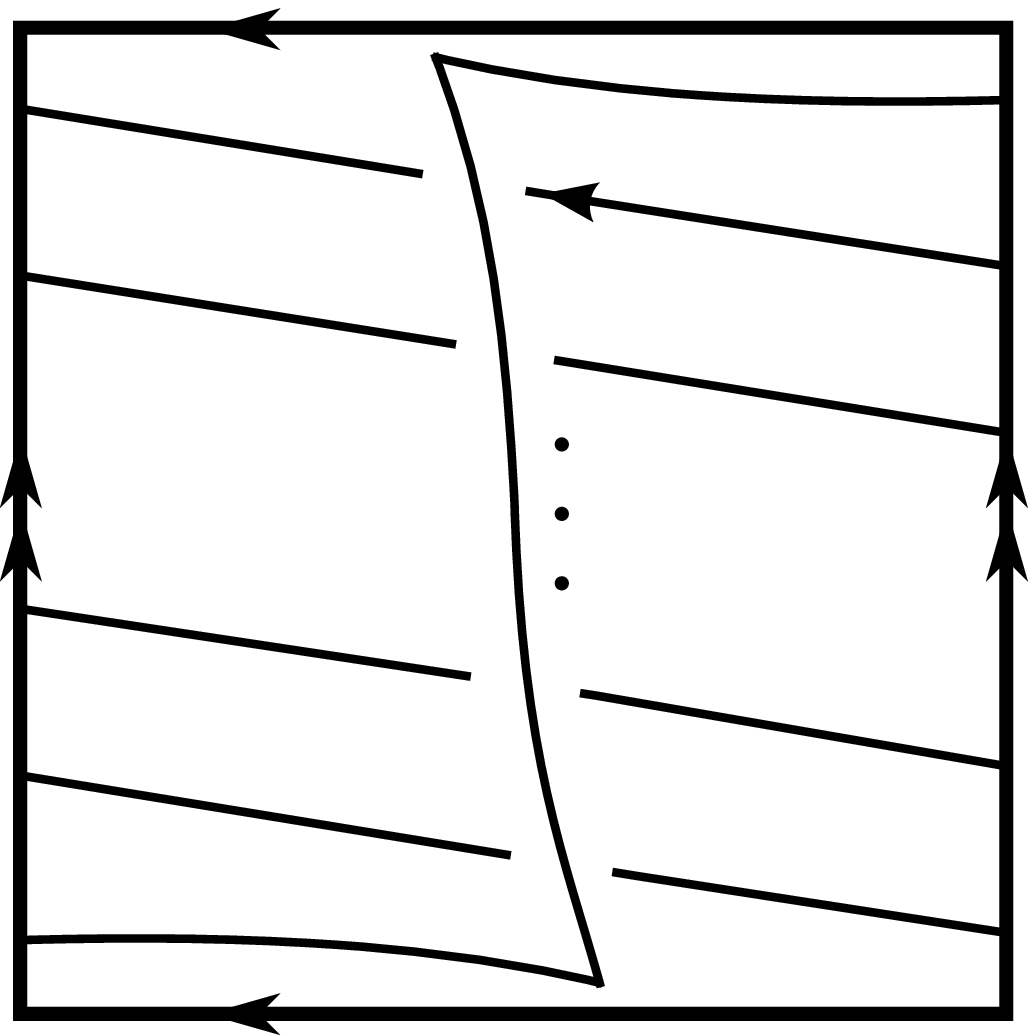}
\caption{ }
\label{f:Kn1cusp}
\end{minipage}
\end{figure}

\begin{rmk}
To see how to compute the classical Legendrian knot invariants from a diagram like in Figure~\ref{f:Kn1}, we describe what happens to the rotation number. For a Legendrian knot $K$ in a contact 3-manifold, and $v$ a nonvanishing vector field in the contact planes, one can define the rotation number $rot_v(K) = rot(K)$, as the signed number of times the tangent vector field of $K$ rotates, relative to $v$, in the contact planes \cite{Gompf1}. This number is independent of the choice of the nonvanishing vector field $v$. In the presentations of knots in $S^1 \times S^2$, by their front projections in $\R^2 / 2\pi \Z^2$ (and knots in \text{standard form}), we can choose $v$ to be $\frac{\partial}{\partial X}$ inside the square (or box). This corresponds to computing $rot(K)$ with counting cusps, as in (\ref{eq:rot}). However, when we extend the vector field $\frac{\partial}{\partial X}$ to a nonvanishing vector field on all of $S^1 \times S^2$, then the latter vector field will make a $360^{\circ}$ twist going from the top edge of the square, $\R^2 / 2\pi \Z^2$, to the bottom. Consequently, one can compute the rotation number of a Legendrian knot in $\R^2 / 2\pi \Z^2$ by counting the cusps as in equation (\ref{eq:rot}) and adding to that $\pm$ the number of times the knot crosses over from the top to the bottom edge of the square.
\end{rmk}

As a result, both $(N({\L_{n,1}^{\sharp}}), \omega |_{N(\L_{n,1}^{\sharp})})$ and $(B_n,\omega_n)$ can be represented by the same Kirby-Stein diagram, i.e. Figure~\ref{f:bnsphstein}. Thus, there exists a symplectomorphism between the symplectic completions of these two manifolds. \end{proof}

Lemma~\ref{l:srbu3} implies that for a small enough $\lambda$, ($\lambda << 1$), we can find a symplectomorphic copy of $(B_n,\omega_n)$ in $(X,\omega)$ as follows: let $\iota$ be the identification of the copy of $(N(\L_{n,1}^{\sharp}), \omega |_{N(\L_{n,1}^{\sharp})})$ in $(N(\L_{n,1}^{\sharp}), \omega |_{N(\L_{n,1}^{\sharp})})^+$ to the copy of $(N(\L_{n,1}^{\sharp}), \omega |_{N(\L_{n,1}^{\sharp})})$ in $(X, \omega)$, then we have an embedding:
\begin{equation}
\iota \circ f^{-1}(B_n, \lambda \omega_n) \hookrightarrow (X,\omega)
\end{equation}
%\vspace{.1in}

\noindent where $f$ is the symplectomorphism in (\ref{eq:sympplus}). As a consequence, combining the results of Lemmas~\ref{l:srbu1},~\ref{l:srbu2} and~\ref{l:srbu3}, we have shown that for each $n \geq 2$, if there exists a Lagrangian core $\L_{n,1} \subset (X,\omega)$, then for a small enough $\lambda$, there exists an embedding of the rational homology ball: $(B_n,\lambda \omega_n) \hookrightarrow (X,\omega)$; hence proving the first part of Theorem~\ref{thm:srbu}. Note, as stated before, just like the symlectic blow-up, the symplectic rational blow-up operation is unique up to the choice of volume of the rational homology ball $B_n$, i.e. the choice of a $\lambda$ that works for this construction. 

\subsection{Gluing argument using the contact manifolds on the boundaries}

In the final step of our proof of Theorem~\ref{thm:srbu}, we will show using Proposition~\ref{p:iso3} (proved in section~\ref{sec:auxprop}), that we can symplectically rationally blow-up $(X,\omega)$ by removing $(B_n,\lambda_0 \omega_n)$ and replacing it with $(C_n,\mu \omega_n')$, for some $\lambda_0 < \lambda$ and $\mu > 0$.

\begin{prop}
\label{p:iso3}
Let $(\partial B_n, \xi) = \partial (B_n,\omega_n)$ and $(\partial C_n,\xi') = \partial (C_n, \omega_n')$ be contact manifolds and let $\xi_{std}$ be the standard contact structure on $(L(n^2,n-1)$. We have $(\partial B_n, \xi)$ $\cong (L(n^2,n-1), \xi_{std})$ $\cong (\partial C_n, \xi')$ as contact 3-manifolds. In particular, this implies that $(B_n,\omega_n)$ is a symplectic filling of $(L(n^2,n-1),\xi_{std})$.
\end{prop} 

\begin{figure}[ht!]
\labellist
\small\hair 2pt
\pinlabel $(B_n,\omega_n)^+$ at 85 224
\pinlabel $h(\A)$ at 135 67
\pinlabel $(C_n,\omega_n')^+$ at 225 222
\pinlabel $\A$ at 130 143
\pinlabel $CN(\partial (C_n,\mu\omega_n'))$ at 300 65
\pinlabel $g$ at 152 97
\pinlabel $h$ at 85 85
\endlabellist
\centering
\includegraphics[height=75mm, width=125mm]{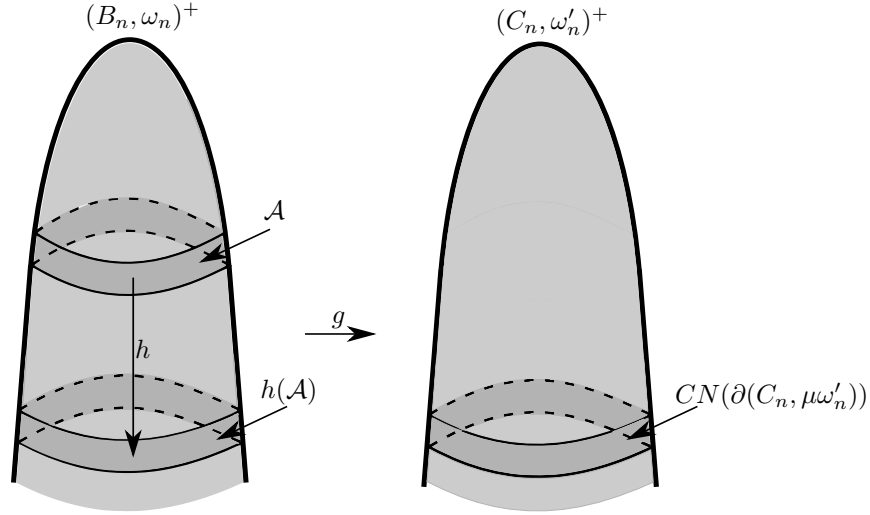}
\caption{Symplectic completions of $(B_n,\omega_n)$ and $(C_n,\omega_n')$}
\label{f:bncnplus}
\end{figure}

We start by assuming that we have $\L_{n,1} \subset (X,\omega)$, implying that we can find an embedding  $(B_n,\lambda \omega_n) \hookrightarrow (X,\omega)$. According to Proposition~\ref{p:iso3}, $\partial (B_n,\omega_n) \cong (L(n^2,n-1), \xi_{std}) \cong \partial (C_n, \omega_n')$, thus for some high enough $t$ we will have a  symplectomorphism:
\begin{equation}
g: [t, \oo) \times \partial(B_n,\omega_n) \rightarrow [t, \oo) \times \partial(C_n,\omega_n')
\end{equation} 
\noindent such that
\begin{eqnarray*}
\left[t, \oo \right) \times \partial(B_n,\omega_n) &\subset& (B_n, \omega_n)^+  \\
\left[t, \oo\right) \times \partial(C_n,\omega_n') &\subset& (C_n, \omega_n')^+ 
\end{eqnarray*} 
%\vspace{.1in}

\noindent where $(B_n, \omega_n)^+$ and $(C_n, \omega_n')^+$ are the symplectic completions of $(B_n, \omega_n)$ and $(C_n,$ $\omega_n')$ respectively. We take the embedding $(B_n,\lambda \omega_n) \hookrightarrow (X,\omega)$, and consider its image $f \circ \iota^{-1}(B_n,\lambda \omega_n)$ back in $(B_n, \omega_n)^+$. Likewise, for $\lambda_0 < \lambda$, we can consider the image of $f \circ \iota^{-1}(B_n,\lambda_0 \omega_n)$ in $(B_n, \omega_n)^+$. We define the $\A \subset (B_n, \omega_n)^+$ to be: 
\begin{equation}
\A = (f \circ \iota^{-1}(B_n,\lambda \omega_n)) - (f \circ \iota^{-1}(B_n,\lambda_0 \omega_n)) 
\end{equation}
%\vspace{.1in}

\noindent so that $\A$ is a collar neighborhood of the boundary of $f \circ \iota^{-1}(B_n,\lambda \omega_n)$.

\begin{figure}[ht!]
\labellist
\small\hair 2pt
\pinlabel $\L_{n,1}$ at 268 190
\pinlabel $(B_n,\lambda_0\omega_n)$ at 250 245
\pinlabel $(X,\omega)$ at 75 152
\pinlabel $(X,\omega)-(B_n,\lambda_0\omega_n)$ at 70 27
\pinlabel $(C_n,\mu\omega_n')$ at 365 27
\pinlabel $\A$ at 212 138
\pinlabel $CN(\partial (C_n,\mu\omega_n'))$ at 300 130
\pinlabel $\phi$ at 255 50
\endlabellist
\centering
\includegraphics[height=95mm, width=125mm]{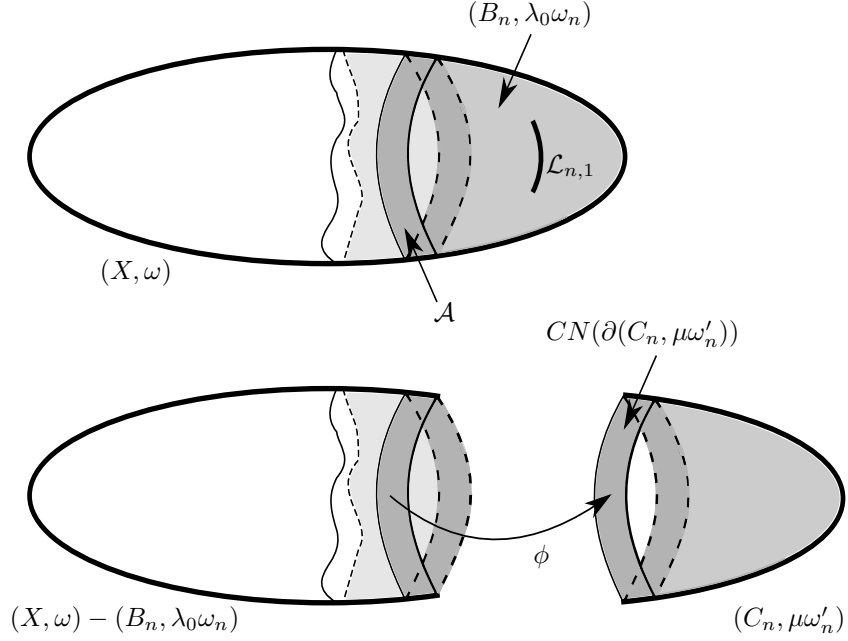}
\caption{Construction on $(X',\omega')$}
\label{f:sympglue}
\end{figure}

We let $h$ be the symplectomorphism corresponding to a radial vector field flow in $(B_n, \omega_n)^+$, then we can find a $\mu > 0$ such that $\A \subset (B_n, \omega_n)^+$ is symplectomorphic to $g \circ h (\A) \cong CN(\partial (C_n,\mu \omega_n')) \subset (C_n, \omega_n')^+$, where $CN(\partial (C_n,\mu \omega_n'))$ denotes a collar neighborhood of $\partial (C_n,\mu \omega_n')$ in $(C_n, \omega_n')^+$ (see Figure~\ref{f:bncnplus}).

Finally, we are ready to construct the \textit{symplectic rational blow-up} $(X',\omega')$ of $(X,\omega)$ (see Figure~\ref{f:sympglue}). We let:
\begin{equation}
\label{eq:srbueq}
(X',\omega') = ((X,\omega) - (B_n,\lambda_0 \omega_n)) \cup_{\phi} (C_n,\mu \omega_n')
\end{equation}
%\vspace{.1in}

\noindent where $\phi$ is the symplectic map:
\begin{equation*}
\phi : \iota \circ f^{-1}(\A) \rightarrow CN(\partial (C_n,\mu \omega_n')). 
\end{equation*}  \end{proof}

It is worthwhile to note, that given the definition of the \textit{symplectic rational blow-up}, one can ask the following symplectic capacity question: Given $\lambda_0$, what is the upper bound on $\mu$ such that the construction in (\ref{eq:srbueq}) works?

\section{Proof of Proposition~\ref{p:iso3}}
\label{sec:auxprop} 

In this section we will prove Proposition~\ref{p:iso3} using computations of Go\-mpf's invariant introduced in section~\ref{sec:kscalc}. We compute Gompf's $\Gamma$ invariant for $(L(n^2,n-1), \xi_{std})$, (section~\ref{sec:compgammalens}), $\partial (B_n, J_n)$, (section~\ref{sec:gammabn}), and for $\partial (C_n,J_n')$, (section~\ref{sec:gammacn}). Note, by the standard contact structure $\xi_{std}$ on $(L(n^2,n-1)$, we mean the contact structure that descends to $L(n^2,n-1)$ from the standard contact structure on $S^3$, via the identification $L(n^2,n-1) = S^3/G_{n^2,n-1}$, where $G_{n^2,n-1}$ is the subgroup
\begin{equation*}
G_{n^2,n-1} = \left\{ \left(\begin{array}{cc} \zeta & 0 \\ 0 & \zeta^{n-1}\end{array}\right) | \zeta^{n^2} = 1 \right\} \subset U(2) \, .  
\end{equation*} 

\vspace{.05in}

\subsection{Computations of $\Gamma$ for $(L(n^2, n-1), \xi_{std})$} 
\label{sec:compgammalens}

In 2006, Lisca \cite{Lisca} classified all the symplectic fillings of $(L(p,q), \xi_{std})$ up to diffeomorphisms and blow-ups. In order to show that the boundaries of the symplectic 4-manifolds he constructed are the lens spaces with the standard contact structure $(L(p,q), \xi_{std})$, he computed the Gompf invariant $\Gamma$ of $(L(p,q), \xi_{std})$ by expressing the contact manifold as the link of a cyclic quotient singularity. We will use his calculations, in the case of $p=n^2$ and $q=n-1$, to match up to our own calculations of $\Gamma$ for $\partial(B_n,\omega_n)$ and $\partial(C_n, \omega'_n)$.

As mentioned above, $(L(n^2, n-1), \xi_{std})$ can be expressed as a link of a cyclic quotient singularity. There is a canonical resolution of this singularity with an exceptional divisor, with a neighborhood $R_{n^2,n-1}$. Let $l_1 \cup l_2$ be the union of two distinct complex lines in $\mathbb{C}P^2$. After successive blow-ups, we can obtain a string $C$ of rational curves in $\mathbb{C}P^2 \# (n+1)\overline{\C P^2}$ of type $(1,-1, -2, \ldots, -2,-n)$ (with $(n-1)$ of $-2$'s), with $\nu(C)$ a regular neighborhood of $C$. It is shown in (\cite{Lisca}, section 6) that there is a natural orientation preserving diffeomorphism from the complement of $\nu(C)$ to $R_{n^2,n-1}$. The boundary of $\nu(C)$ is an oriented 3-manifold which can be given by a surgery presentation of unknots $U_0$, $\ldots$, $U_{n+1}$ (Figure~\ref{f:nuc}), where $\nu_0, \ldots, \nu_{n+1}$ are the generators of $H_1(\partial \nu(C);\mathbb{Z})$. If the unknot $U_0$ is blown down, we have a natural identification 
\begin{equation}
\nu(C) = -L(n^2,n-1) =L(n^2,n^2-(n-1))
\end{equation}
since $[-2,-2, \ldots, -2, -n]$, with $n$ amount of $(-2)$s, is the continued fraction expansion of $\frac{n^2}{n^2-(n-1)}$. 

\begin{figure}[ht!]
\labellist
\small\hair 2pt
\pinlabel $1$ at 33 4.7
\pinlabel $-1$ at 60 4.7
\pinlabel $-2$ at 90 4.7
\pinlabel $-2$ at 173 4.7
\pinlabel $-n$ at 203 4.7
\pinlabel $U_0$ at 32 2.5 
\pinlabel $U_1$ at 62 2.5 
\pinlabel $U_2$ at 92 2.5 
\pinlabel $U_{n}$ at 175 2.5 
\pinlabel $U_{n+1}$ at 205 2.5 
\endlabellist
\centering
\includegraphics[height=30mm, width=120mm]{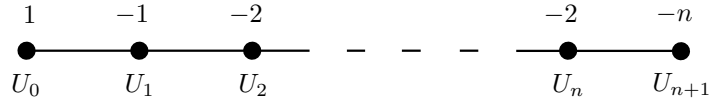}
\caption{{\bf Surgery diagram of $\partial \nu(C)$}}
\label{f:nuc}
\end{figure}

The relations of $\nu_0, \ldots, \nu_{n+1}$ in $H_1(\partial \nu(C);\mathbb{Z})$ are as follows:

\[
\left.\begin{array}{l}
    \nu_0 + \nu_1 = 0 \\ \nu_0 - \nu_1 + \nu_2 = 0 \\ \nu_1 - 2\nu_2 + \nu_3 = 0 \\ \nu_2 - 2\nu_3 + \nu_4 = 0 \\ \vdots \\ \nu_{n-1} - 2\nu_{n} + \nu_{n+1} = 0 \\ \nu_{n} - n\nu_{n+1} = 0 
\end{array}\right\} 
\implies 
\left.\begin{array}{l}
    \nu_0 = -\nu_1 \\ \nu_2 = 2\nu_1 \\ \nu_3 = 3\nu_1 \\ \nu_4 = 4\nu_1 \\ \vdots \\ \nu_{n+1} = (n+1)\nu_1 \\ (n^2)\nu_1 = 0 \, .
\end{array}\right.
\]

\vspace{.1in}

Lisca applied a slight generalization of Theorem~\ref{thm:g3} (\cite{Lisca}, Theorem 6.2), and computed the value of Gompf's $\Gamma$ invariant of $\partial \nu(C)=-L(p,q)$. For our purposes we restate it with $p=n^2$ and $q=n-1$, and we will handle the even and odd values of $n$ seperately.

For $n$ odd, the lens spaces $L(n^2, n^2 - (n-1))$ each have one spin structure $\mathfrak{t}$, which can be specified by the characteristic sublink $L(\mathfrak{t}) = U_1 \cup U_3 \cup U_5 \cup \ldots \cup U_n$. Also, $L_0 = \emptyset$, (see equation~\ref{eq:gamma})). Consequently, we have:
\begin{eqnarray*}
PD\Gamma_{L(n^2,n-1)}(\xi_{std}, \mathfrak{t}) &=& -PD\Gamma_{L(n^2,n^2 - (n-1))}(\xi_{std}, \mathfrak{t})        \\
&=& - \nu_0 - \nu_1 + \nu_2 - \nu_3 + \ldots - \nu_{n} + \frac{n-1}{2}\nu_{n+1}   \\
&=& \nu_1 - \nu_1 + 2\nu_1 - 3\nu_1 + \dots -n\nu_1 + \frac{n^2 - 1}{2}\nu_1   \\
&\equiv& \frac{n^2 - n}{2}\nu_1  \mod{n^2} \, .
\end{eqnarray*}

%\vspace{.1in}

For $n$ even, the lens spaces $L(n^2, n^2 - (n-1))$ each have two spin structures $\mathfrak{t}_1$ and $\mathfrak{t}_2$, corresponding to the characteristic sublinks $L(\mathfrak{t}_1) = U_0$ and $L(\mathfrak{t}_2) = U_1 \cup U_3 \cup \ldots \cup U_{n+1}$ respectively. As before, $L_0 = \emptyset$. Consequently we have:
\begin{eqnarray*}
PD\Gamma_{L(n^2,n-1)}(\xi_{std}, \mathfrak{t}_1) &=& -PD\Gamma_{L(n^2,n^2 - (n-1))}(\xi_{std}, \mathfrak{t}_1)       \\
&=& - \nu_0 + \nu_{n+1}   \\
&=& \nu_1 + \frac{n^2 - n - 2}{2}\nu_1   \\
&\equiv& \frac{n^2 - n}{2}\nu_1  \mod{n^2}   \\
PD\Gamma_{L(n^2,n-1)}(\xi_{std}, \mathfrak{t}_2) &=& -PD\Gamma_{L(n^2,n^2 - (n-1))}(\xi_{std}, \mathfrak{t}_2)   \\     
&=& - \nu_0 - \nu_1 + \nu_2 - \nu_3 + \ldots + \nu_{n+1}  \\
&=& \nu_1 - \nu_1 + 2\nu_1 -3\nu_1 + 4\nu_1 - \ldots -(n+1)\nu_1   \\
&\equiv& \frac{n}{2}\nu_1  \mod{n^2} \, .
\end{eqnarray*}

%\vspace{.1in}

\subsection{Computations of $\Gamma$ for $\partial (B_n, J_n)$}
\label{sec:gammabn}

Having described the Stein structure $J_n$ on $B_n$ in section~\ref{sec:sympbncn}, we are ready to compute the $\Gamma$ invariant of $\partial (B_n,J_n) = (\partial B_n, \xi)$ where $\xi$ is the induced contact structure, $\xi = T\partial B_n \cap JT\partial B_n$. As described in Theorem~\ref{thm:g3}, we construct the manifold $B_n^*$ from $B_n$, where we replace the 1-handle in $B_n$ with a 2-handle attached to an unknot with framing $0$. A diagram for $B_n^*$ is seen in Figure~\ref{f:bnstar}.

\begin{figure}[ht!]
\labellist
\small\hair 2pt
\pinlabel $0$ at 465 240
\pinlabel $-n$ at 235 160
\pinlabel $-n-1$ at -35 240
\pinlabel $K_2^n$ at 120 -10
\pinlabel $K_1^n$ at 360 -10
\endlabellist
\centering
\includegraphics[scale=0.25]{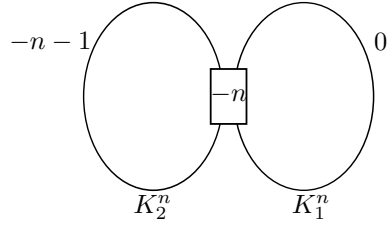}
\caption{{\bf Kirby diagram of $B_n^*$}}
\label{f:bnstar}
\end{figure}

Let $\mu_1$ and $\mu_2$ be the meridians of the knots $K_1^n$ and $K_2^n$, as depicted in Figure~\ref{f:bnstar}. Let  $\alpha_1, \alpha_2$ be the basis of $H_2(B_n^*;\mathbb{Z})$ determined by $K_1^n$ and $K_2^n$. By definition, we have $rot(K_1^n) = 0$, and according to the Stein structure $J_n$, we have $rot(K_1^n) = 1$. The relations of $\mu_1$ and $\mu_2$ in $H_1(\partial B_n;\mathbb{Z})$ are:

\[
\left.\begin{array}{l}
    -n\mu_2 = 0 \\ -n\mu_1 - (n+1)\mu_2 = 0 
\end{array}\right\} 
\implies 
\left.\begin{array}{l}
    \mu_2 = -n\mu_1 \\ (n^2)\mu_1 = 0 \, .
\end{array}\right.
\]

\vspace{.1in}

For $n$ odd, as before, $\partial B_n = L(n^2,n-1)$ has only one spin structure, $\mathfrak{s}$, whose characteristic sublink is $L(\mathfrak{s}) = \emptyset$. Additionally, we have $L_0 = K_1^n$. Letting $\rho$ be as in Theorem~\ref{thm:g3}, we have:
\begin{eqnarray*}
\left\langle \rho, \alpha_1 \right\rangle &=& \frac{1}{2}(rot(K_1^n) + \ell k(K_1^n, K_1^n)) = \frac{1}{2}(0 + 0) = 0   \\
\left\langle \rho, \alpha_2 \right\rangle &=& \frac{1}{2}(rot(K_2^n) + \ell k(K_2^n, K_1^n)) = \frac{1}{2}(1-n) \, .
\end{eqnarray*}

%\vspace{.1in}

\noindent Using the above, we compute $PD\Gamma_{\partial B_n}(\xi,\mathfrak{s})$:  
\begin{eqnarray*}
PD\Gamma_{\partial B_n}(\xi,\mathfrak{s}) &=& \left\langle \rho, \alpha_1 \right\rangle \mu_1 + \left\langle \rho, \alpha_2 \right\rangle \mu_2  \\
&=& 0\mu_1 + \frac{1-n}{2}\mu_2   \\
&\equiv& \frac{n^2-n}{2} \mu_1 \mod{n^2} \, .
\end{eqnarray*}

%\vspace{.1in}

For $n$ even, $\partial B_n = L(n^2,n-1)$ has two spin structures $\mathfrak{s}_1$ and $\mathfrak{s}_2$, corresponding to the characteristic sublinks $L(\mathfrak{s}_1) = K_2^n$ and $L(\mathfrak{s}_2) = K_1^n + K_2^n$ respectively, (and $L_0 = K_1^n$ as before). We have for the spin structure $\mathfrak{s}_1$:
\begin{eqnarray*}
\left\langle \rho, \alpha_1 \right\rangle &=& \frac{1}{2}(rot(K_1^n) + \ell k(K_1^n, K_1^n + K_2^n)) = \frac{1}{2}(0-n) = \frac{-n}{2}   \\
\left\langle \rho, \alpha_2 \right\rangle &=& \frac{1}{2}(rot(K_2^n) + \ell k(K_2^n, K_1^n + K_2^n)) = \frac{1}{2}(1-(2n+1)) = -n \, .  
\end{eqnarray*}

%\vspace{.1in}

\noindent Therefore,
\begin{eqnarray*}
PD\Gamma_{\partial B_n}(\xi,\mathfrak{s}_1) &=& \left\langle \rho, \alpha_1 \right\rangle \mu_1 + \left\langle \rho, \alpha_2 \right\rangle \mu_2   \\
&=& \frac{-n}{2}\mu_1 -n\mu_2   \\
&\equiv& \frac{2n^2-n}{2} \mu_1 \mod{n^2} \, .
\end{eqnarray*}

%\vspace{.1in}

For the spin structure $\mathfrak{s}_2$ we get:
\begin{eqnarray*}
\left\langle \rho, \alpha_1 \right\rangle &=& \frac{1}{2}(rot(K_1^n) + \ell k(K_1^n, 2K_1^n + K_2^n)) = \frac{1}{2}(0-n) = \frac{-n}{2}    \\
\left\langle \rho, \alpha_2 \right\rangle &=& \frac{1}{2}(rot(K_2^n) + \ell k(K_2^n, 2K_1^n + K_2^n)) = \frac{1}{2}(1-(3n+1)) = \frac{-3n}{2} \, .  
\end{eqnarray*}

%\vspace{.1in}

\noindent Therefore,
\begin{eqnarray*}
PD\Gamma_{\partial B_n}(\xi,\mathfrak{s}_2) &=& \left\langle \rho, \alpha_1 \right\rangle \mu_1 + \left\langle \rho, \alpha_2 \right\rangle \mu_2   \\
&=& \frac{-n}{2}\mu_1 - \frac{3n}{2}\mu_2   \\
&\equiv& \frac{n^2-n}{2} \mu_1 \mod{n^2} \, .
\end{eqnarray*}

%\vspace{.1in}

\subsection{Computations of $\Gamma$ for $\partial (C_n, J'_n)$}
\label{sec:gammacn}

Next, we will compute the $\Gamma$ invariant for $\partial (C_n, J'_n)$, where $J'_n$ is a Stein structure on $C_n$, described in section~\ref{sec:sympbncn}. This Stein structure induces $\xi' = T\partial C_n \cap JT\partial C_n$, the contact structure on the boundary $\partial (C_n, J'_n) = (\partial C_n, \xi')$.

%\vspace{.2in}

Let $\lambda_1, \ldots, \lambda_{n-1}$ be the meridians of the knots $W_1, \ldots, W_{n-1}$, as in Figure~\ref{f:cnstein}. Also, let  $\beta_1, \ldots, \beta_{n-1}$ be the basis of $H_2(C_n;\mathbb{Z})$ determined by $W_1, \ldots, W_{n-1}$. The relations of $\lambda_1, \ldots, \lambda_{n-1}$ in $H_1(\partial C_n;\mathbb{Z})$ are as follows:
 
\[
\left.\begin{array}{l}
    (-n-2)\lambda_1 + \lambda_2 = 0 \\ \lambda_1 -2\lambda_2 + \lambda_3 = 0 \\ \lambda_2 -2\lambda_3 + \lambda_4 = 0 \\ \vdots \\ \lambda_{n-3} -2\lambda_{n-2} +\lambda_{n-1} = 0 \\ \lambda_{n-2} -2\lambda_{n-1} = 0
\end{array}\right\} 
\implies 
\left.\begin{array}{l}
    \lambda_2 = (n+2)\lambda_1 \\ \lambda_3 = (2n+3)\lambda_1 \\ \lambda_4 = (3n+4)\lambda_1 \\ \vdots \\ \lambda_{n-1} = (n^2-n-1)\lambda_1 \\ (n^2)\lambda_1 = 0 \, .
\end{array}\right.
\]

%\vspace{.1in}

As before, for $n$ odd, $\partial C_n = L(n^2,n-1)$ has only one spin structure, $\mathfrak{r}$, represented by the characteristic sublink $L(\mathfrak{r}) = W_2 + W_4 + W_6 + \cdots + W_{n-1}$, in addition, we have $L_0 = \emptyset$. Again, letting $\rho$ be as in Theorem~\ref{thm:g3}, we have:
\begin{eqnarray*}
\left\langle \rho, \beta_1 \right\rangle &=& \frac{1}{2}(rot(W_1) + \ell k(W_1, W_2 + W_4+ \cdots + W_{n-1})) = \frac{1}{2}(-n + 1) = \frac{1-n}{2}   \\
\left\langle \rho, \beta_2 \right\rangle &=& \frac{1}{2}(rot(W_2) + \ell k(W_2, W_2 + W_4+ \cdots + W_{n-1})) = \frac{1}{2}(0 - 2) = -1   \\
\left\langle \rho, \beta_3 \right\rangle &=& \frac{1}{2}(rot(W_3) + \ell k(W_3, W_2 + W_4+ \cdots + W_{n-1})) = \frac{1}{2}(0 + 2) = 1   \\
&\vdots&  \\
\left\langle \rho, \beta_{n-2} \right\rangle &=& \frac{1}{2}(rot(W_{n-2}) + \ell k(W_{n-2}, W_2 + W_4+ \cdots + W_{n-1})) = \frac{1}{2}(0 + 2) = 1   \\
\left\langle \rho, \beta_{n-1} \right\rangle &=& \frac{1}{2}(rot(W_{n-1}) + \ell k(W_{n-1}, W_2 + W_4+ \cdots + W_{n-1})) = \frac{1}{2}(0 - 2) = -1  \, .  
\end{eqnarray*}

%\vspace{.1in}

Using the above, we compute $PD\Gamma_{\partial C_n}(\xi',\mathfrak{r})$:  
\begin{eqnarray*}
PD\Gamma_{\partial C_n}(\xi',\mathfrak{r}) &=& \left\langle \rho, \beta_1 \right\rangle \lambda_1 + \cdots + \left\langle \rho, \beta_{n-1} \right\rangle \lambda_{n-1}   \\
&=& \frac{1-n}{2}\lambda_1 - \lambda_2 + \lambda_3 - \cdots + \lambda_{n-2} - \lambda_{n-1}    \\
&=& \frac{1-n}{2}\lambda_1 - (n+2)\lambda_1 + (2n+3)\lambda_1 - \cdots - (n^2-n-1)\lambda_1  \\
&=& \frac{1-n}{2}\lambda_1 - \frac{n^2 +1}{2}\lambda_1   \\
&\equiv& \frac{n^2-n}{2}\lambda_1  \mod{n^2} \, .
\end{eqnarray*}

%\vspace{.1in}

For $n$ even, $\partial C_n = L(n^2,n-1)$ has two spin structures $\mathfrak{r}_1$ and $\mathfrak{r}_2$, corresponding to the characteristic sublinks $L(\mathfrak{r}_1) = W_1 + W_3 + W_5 + \cdots + W_{n-1}$ and $L(\mathfrak{r}_2) = \emptyset $ respectively, (and $L_0 = \emptyset$ as before). For the spin structure $\mathfrak{r}_1$, we have:
\begin{eqnarray*}
\left\langle \rho, \beta_1 \right\rangle &=& \frac{1}{2}(rot(W_1) + \ell k(W_1, W_1 + W_3+ \cdots + W_{n-1})) = -(n+1)   \\
\left\langle \rho, \beta_2 \right\rangle &=& \frac{1}{2}(rot(W_2) + \ell k(W_2, W_1 + W_3+ \cdots + W_{n-1})) = 1   \\
\left\langle \rho, \beta_3 \right\rangle &=& \frac{1}{2}(rot(W_3) + \ell k(W_3, W_1 + W_3+ \cdots + W_{n-1})) = -1   \\
&\vdots&  \\
\left\langle \rho, \beta_{n-2} \right\rangle &=& \frac{1}{2}(rot(W_{n-2}) + \ell k(W_{n-2}, W_1 + W_3+ \cdots + W_{n-1})) = 1   \\
\left\langle \rho, \beta_{n-1} \right\rangle &=& \frac{1}{2}(rot(W_{n-1}) + \ell k(W_{n-1}, W_1 + W_3+ \cdots + W_{n-1})) = -1  \, . 
\end{eqnarray*}

%\vspace{.1in}

\noindent Therefore,
\begin{eqnarray*}
PD\Gamma_{\partial C_n}(\xi',\mathfrak{r_1}) &=& \left\langle \rho, \beta_1 \right\rangle \lambda_1 + \cdots + \left\langle \rho, \beta_{n-1} \right\rangle \lambda_{n-1}  \\
&=& -(n+1)\lambda_1 + \lambda_2 - \lambda_3 + \cdots + \lambda_{n-2} - \lambda_{n-1}   \\
&=& -(n+1)\lambda_1 + (n+2)\lambda_1 - (2n+3)\lambda_1 + \cdots - (n^2-n-1)\lambda_1  \\
&\equiv& \frac{n^2-n}{2}\lambda_1  \mod{n^2} \, .
\end{eqnarray*}

%\vspace{.1in}

For the spin structure $\mathfrak{r}_2$ we get:
\begin{eqnarray*}
\left\langle \rho, \beta_1 \right\rangle &=& \frac{1}{2}(rot(W_1) + \ell k(W_1, \emptyset)) = \frac{-n}{2}    \\
\left\langle \rho, \beta_2 \right\rangle &=& \frac{1}{2}(rot(W_2) + \ell k(W_2, \emptyset)) = 0    \\
&\vdots&  \\
\left\langle \rho, \beta_{n-1} \right\rangle &=& \frac{1}{2}(rot(W_{n-1}) + \ell k(W_{n-1}, \emptyset)) = 0  \, . 
\end{eqnarray*}

%\vspace{.1in}

\noindent Therefore,
\begin{eqnarray*}
PD\Gamma_{\partial C_n}(\xi',\mathfrak{r_2}) &=& \left\langle \rho, \beta_1 \right\rangle \lambda_1 + \cdots + \left\langle \rho, \beta_{n-1} \right\rangle \lambda_{n-1}   \\
&=& \frac{-n}{2}\lambda_1   \\
&\equiv& \frac{2n^2-n}{2}\lambda_1  \mod{n^2} \, .
\end{eqnarray*}
%\vspace{.1in}

\subsection{Showing $(\partial B_n, \xi) \cong (L(n^2,n-1), \xi_{std}) \cong (\partial C_n, \xi')$}

Finally, we are ready to prove Propostion~\ref{p:iso3}, since we computed the $\Gamma$ invariant for the manifolds $(\partial B_n, \xi)$, $(L(n^2,n-1), \xi_{std})$, and $(\partial C_n, \xi')$. In order to show these manifolds have the same contact structure, ($\xi \cong \xi_{std} \cong \xi'$), we have to find a suitable identification between these manifolds, in particular between their first homology groups. It is important to note that the contact structures $\xi$ and $\xi'$ are tight \cite{Lisca}, since they were induced from the boundaries of Stein surfaces. Therefore, due to the classification of tight contact structures on lens spaces $L(p,q)$ \cite{Gir} \cite{Ho}, the $\Gamma$ invariant is sufficient to show the isomorphisms between these contact 3-manifolds. This is because the $\Gamma$ invariant shows which $spin^c$ structures are induced by the contact structures $\xi$, $\xi'$, and $\xi_{std}$, since $\Gamma(\zeta,\cdot):Spin(M) \rightarrow H_1(M;\mathbb{Z})$ depends only on the homotopy class $[\zeta]$.      

Figure~\ref{f:kirbybntostd} demonstrates a sequence of Kirby calculus moves from $\partial B_n^*$ to $-\partial \nu(C)$, (compare with Figure~\ref{f:nuc}). For a detailed account of Kirby calculus, see \cite{GS}. (Note, for shorthand we represent most spheres by dots, as in Figure~\ref{f:cn}.) As the moves are performed, we keep track of the $\mu_i \in H_1(\partial B_n^*;\mathbb{Z})$, the meridians of the associated unknots in the diagram. In move $\mathbf{I}$ we perform $n$ blow-ups. In moves $\mathbf{II}$ and $\mathbf{III}$ we perform a handleslide. In moves $\mathbf{IV}_1, \ldots, \mathbf{IV}_{n-3}$ we perform a handleslide in each. Finally, in move $\mathbf{V}$, we blow-down the unknot with framing $(-1)$.

As a result we can form the following identifications between $\mu_1, \mu_2 \in H_1(\partial B_n;\mathbb{Z})$ and $\nu_0, \ldots, \nu_{n+1} \in H_1(L(n^2,n-1);\mathbf{Z})$:
\begin{eqnarray}
\label{eq:munu}
\mu_1 &=& \nu_{n+1} \nonumber \\
n\mu_1 + n\mu_2 &=& \nu_n \nonumber \\
(n-1)\mu_1 + (n-1)\mu_2 &=& \nu_{n-1} \nonumber \\
&\vdots& \nonumber \\
2\mu_1 + 2\mu_2 &=& \nu_2 \nonumber \\
\mu_1 + \mu_2 &=& \nu_1 \, .
\end{eqnarray}

\begin{figure}[ht!]
\labellist
\small\hair 2pt 
\pinlabel $-n-1$ at 80 390
\pinlabel $0$ at 190 390
\pinlabel $-n$ at 130 315
\pinlabel $\mu_2$ at 80 240
\pinlabel $\mu_1$ at 190 240
\pinlabel $\mathbf{I}$ at 280 268 
\pinlabel $n$ at 320 333 
\pinlabel $-1$ at 450 333 
\pinlabel $1$ at 370 385
\pinlabel $1$ at 380 310 
\pinlabel $1$ at 380 280 
\pinlabel $1$ at 370 233 
\pinlabel $\mu_1$ at 320 303
\pinlabel $\mu_2$ at 450 303
\pinlabel $\mu_1+\mu_2$ at 425 385
\pinlabel $\mu_1+\mu_2$ at 425 233
\pinlabel $\mathbf{II}$ at 495 268
\pinlabel $n$ at 515 333 
\pinlabel $-1$ at 645 333 
\pinlabel $1$ at 565 385
\pinlabel $1$ at 575 310 
\pinlabel $1$ at 565 265 
\pinlabel $2$ at 565 233 
\pinlabel $\mu_1$ at 515 303
\pinlabel $\mu_2$ at 645 303
\pinlabel $\mu_1+\mu_2$ at 620 385
\pinlabel $\mu_1+\mu_2$ at 620 233 
\pinlabel $2\mu_1+2\mu_2$ at 630 265 
\pinlabel $\mathbf{III}$ at 15 38
\pinlabel $n$ at 20 110 
\pinlabel $-1$ at 150 110 
\pinlabel $1$ at 70 162
\pinlabel $1$ at 70 70 
\pinlabel $2$ at 70 42 
\pinlabel $2$ at 70 10 
\pinlabel $\mu_1$ at 20 80
\pinlabel $\mu_2$ at 160 95
\pinlabel $\mu_1+\mu_2$ at 125 162
\pinlabel $\mu_1+\mu_2$ at 125 10 
\pinlabel $2\mu_1+2\mu_2$ at 135 40 
\pinlabel $3\mu_1+3\mu_2$ at 135 70 
\pinlabel $\mathbf{IV}_1\ldots\mathbf{IV}_{n-3}$ at 260 85
\pinlabel $n$ at 285 145 
\pinlabel $-1$ at 420 150 
\pinlabel $1$ at 340 147
\pinlabel $2$ at 340 120 
\pinlabel $2$ at 340 55 
\pinlabel $2$ at 340 20 
\pinlabel $\mu_1$ at 285 170
\pinlabel $\mu_2$ at 425 170
\pinlabel $n\mu_1+n\mu_2$ at 355 170
\pinlabel $2\mu_1+2\mu_2$ at 410 53 
\pinlabel $\mu_1+\mu_2$ at 400 20 
\pinlabel $\mathbf{V}$ at 500 90
\pinlabel $n$ at 565 183 
\pinlabel $2$ at 565 147
\pinlabel $2$ at 565 113 
\pinlabel $2$ at 565 40 
\pinlabel $2$ at 565 5 
\pinlabel $\mu_1$ at 598 181
\pinlabel $n\mu_1+n\mu_2$ at 633 147
\pinlabel $2\mu_1+2\mu_2$ at 630 37 
\pinlabel $\mu_1+\mu_2$ at 620 5 
\endlabellist
\centering
\includegraphics[height=80mm, width=125mm]{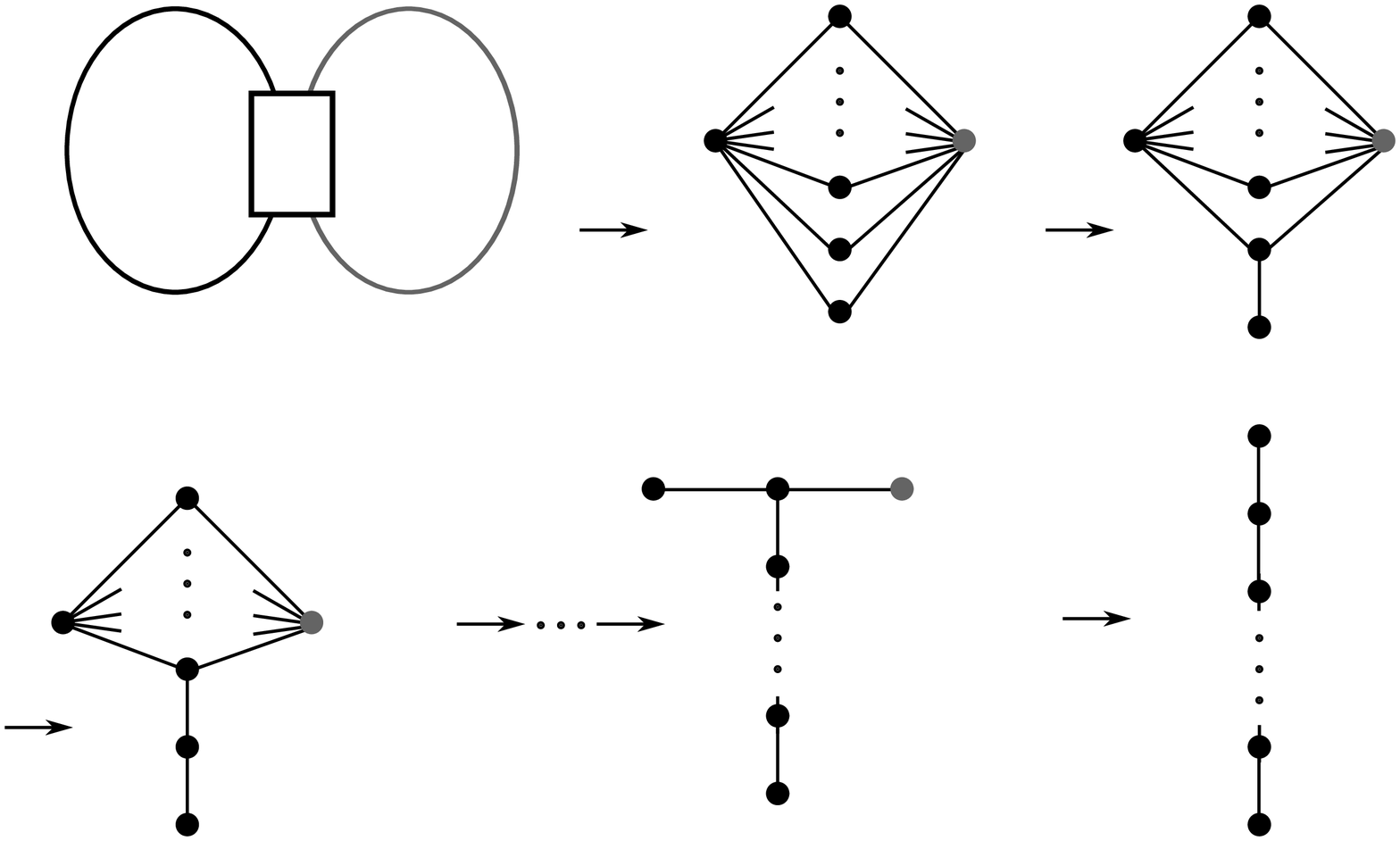}
\caption{{\bf Kirby moves from $\partial B_n^*$ to $-\partial \nu(C)$}}
\label{f:kirbybntostd}
\end{figure}

%\vspace{.1in}

For $n$ odd, $L(n^2,n-1)$ has only one spin structure, so there is no need to keep track of it throughout the Kirby moves. We multiply both sides of the first identification above by $\frac{n^2-n}{2}$ and get:
\begin{eqnarray*}
\frac{n^2-n}{2}\mu_1 &=& \frac{n^2-n}{2}\nu_{n+1} = (n+1)\frac{n^2-n}{2}\nu_1 = \frac{n^3-n}{2}\nu_1  \\
&\equiv& (\frac{n^3-n}{2} - \frac{(n-1)n^2}{2})\nu_1 \equiv \frac{n^2-n}{2}\nu_1 \mod{n^2} \, . 
\end{eqnarray*}

%\vspace{.1in}
\noindent Thus, we have:
\begin{equation*}
PD\Gamma_{(\partial B_n)}(\xi,\mathfrak{s}) = \frac{n^2-n}{2}\mu_1 \equiv \frac{n^2-n}{2}\nu_1 = PD\Gamma_{L(n^2,(n-1))}(\xi_{std},\mathfrak{t}) \, .
\end{equation*}

%\vspace{.1in}

For $n$ even, since $L(n^2,n-1)$ has two spin structures, in addition to matching up the $\mu_i$ to the $\nu_i$, we also also have to make an appropriate identification among the spin structures. In Figure~\ref{f:kirbybntostd} we follow the spin structure $\mathfrak{s}_1$ through the Kirby moves by denoting the knots corresponding to its characteristic sublink in grey color. Thus, we can see that spin structure $\mathfrak{s}_1$ of $\partial B_n$ is identified with the spin structure $\mathfrak{t}_1$ of $-\partial \nu(C)$. If we multiply the first identification of (\ref{eq:munu}) by $\frac{n^2-n}{2}$, we get: 
\begin{eqnarray*}
\frac{n^2-n}{2}\mu_1 &=& \frac{n^2-n}{2}\nu_{n+1} = (n+1)\frac{n^2-n}{2}\nu_1 = \frac{n^3-n}{2}\nu_1  \\
&\equiv& \frac{2n^2-n}{2}\nu_1 \mod{n^2} \, .\nonumber   
\end{eqnarray*}

%\vspace{.1in}

Likewise, if we take the last identification of (\ref{eq:munu}), and apply the relations for $\mu_i$, we get $(1-n)\mu_1 = \nu_1$. We multiply this by $\frac{n^2-n}{2}$, and get:
\begin{equation*}
\frac{n^2-n}{2}\nu_1 = \frac{n^2-n}{2}(1-n)\mu_1 \equiv \frac{2n^2-n}{2}\mu_1 \mod{n^2} \, . 
\end{equation*}

%\vspace{.1in}

\noindent As a result, we have:
\begin{equation}
PD\Gamma_{(\partial B_n)}(\xi,\mathfrak{s}_1) = \frac{2n^2-n}{2}\mu_1 \equiv \frac{n^2-n}{2}\nu_1 = PD\Gamma_{L(n^2,(n-1))}(\xi_{std},\mathfrak{t}_1)
\end{equation}

\begin{equation}
PD\Gamma_{(\partial B_n)}(\xi,\mathfrak{s}_2) = \frac{n^2-n}{2}\mu_1 \equiv \frac{2n^2-n}{2}\nu_1 = PD\Gamma_{L(n^2,(n-1))}(\xi_{std},\mathfrak{t}_2) \, .
\end{equation}

%\vspace{.1in}

\noindent As a consequence, this gives us $(\partial B_n,\xi) \cong (L(n^2,n-1),\xi_{std})$.

In a similar manner, we can show $(\partial B_n, \xi) \cong (\partial C_n, \xi')$. We first find a suitable identification between $\mu_1, \mu_2 \in H_1(\partial B_n;\mathbb{Z})$ and $\lambda_1, \ldots, \lambda_{n-1} \in H_1(\partial C_n;\mathbb{Z})$, by a sequence of Kirby moves depicted in Figure~\ref{f:kirbybntocn}. In move $\mathbf{I}$ we perform a handleslide: we slide $K_1^n$ over $K_2^n$. In moves $\mathbf{II}$ and $\mathbf{III}$ we perform blow-ups. In moves $\mathbf{IV}_1,\ldots,\mathbf{IV}_{n-4}$ we perform a blow-up in each. Finally, in move $\mathbf{V}$ we blow-down the unknot with framing $(1)$. 

In the final diagram we can see a clear identification with Figure~\ref{f:cn2}, which gives us the following:
\begin{eqnarray}
\label{eq:mulambda}
 \mu_1 + \mu_2 &=& \lambda_1 \nonumber \\
2\mu_1 + \mu_2 &=& \lambda_2 \nonumber \\
&\vdots& \nonumber \\
(n-2)\mu_1 + \mu_2 &=& \lambda_{n-2} \nonumber \\
(n-1)\mu_1 + \mu_2 &=& \lambda_{n-1} \, .
\end{eqnarray} 

\begin{figure}[ht!]
\labellist
\small\hair 2pt
\pinlabel $-n-1$ at 70 315
\pinlabel $0$ at 190 315
\pinlabel $-n$ at 130 240
\pinlabel $\mu_2$ at 70 160
\pinlabel $\mu_1$ at 190 160
\pinlabel $\mathbf{I}$ at 280 225 
\pinlabel $-n-1$ at 330 260
\pinlabel $n-1$ at 335 215
\pinlabel $\mu_1+\mu_2$ at 410 260
\pinlabel $\mu_2$ at 390 215
\pinlabel $\mathbf{II}$ at 420 225
\pinlabel $-n-2$ at 465 280
\pinlabel $-1$ at 485 235
\pinlabel $n-2$ at 475 190
\pinlabel $\mu_1+\mu_2$ at 550 280
\pinlabel $2\mu_1+\mu_2$ at 555 235
\pinlabel $\mu_2$ at 530 190
\pinlabel $\mathbf{III}$ at 620 225
\pinlabel $-n-2$ at 635 300 
\pinlabel $-2$ at 655 258
\pinlabel $-1$ at 655 215
\pinlabel $n-3$ at 642 170
\pinlabel $\mu_1+\mu_2$ at 720 300
\pinlabel $2\mu_1+\mu_2$ at 722 257
\pinlabel $3\mu_1+\mu_2$ at 722 212
\pinlabel $\mu_2$ at 695 170
\pinlabel $\mathbf{IV}_1\ldots\mathbf{IV}_{n-4}$ at 85 85
\pinlabel $-n-2$ at 230 115 
\pinlabel $-2$ at 315 115
\pinlabel $-2$ at 390 115
\pinlabel $-1$ at 510 115
\pinlabel $1$ at 590 115
\pinlabel $\mu_1+\mu_2$ at 210 85
\pinlabel $2\mu_1+\mu_2$ at 300 85
\pinlabel $3\mu_1+\mu_2$ at 390 85
\pinlabel $(n-1)\mu_1+\mu_2$ at 510 85
\pinlabel $\mu_2$ at 600 85
\pinlabel $\mathbf{V}$ at 115 10
\pinlabel $-n-2$ at 230 35 
\pinlabel $-2$ at 315 35
\pinlabel $-2$ at 390 35
\pinlabel $-2$ at 510 35
\pinlabel $-2$ at 590 35
\pinlabel $\mu_1+\mu_2$ at 210 5
\pinlabel $2\mu_1+\mu_2$ at 300 5
\pinlabel $3\mu_1+\mu_2$ at 390 5
\pinlabel $(n-1)\mu_1+\mu_2$ at 600 5
\endlabellist
\centering
\includegraphics[height=60mm, width=120mm]{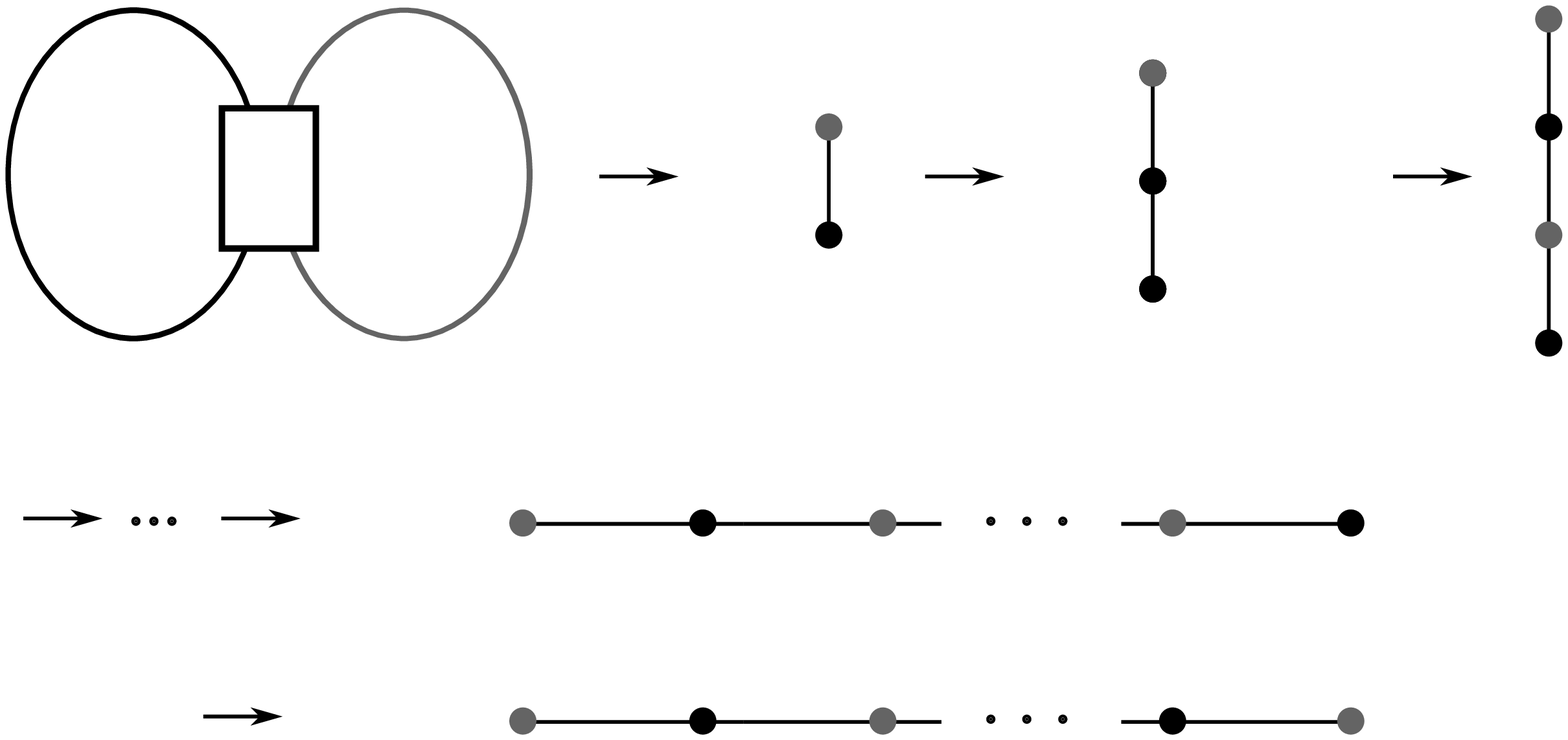}
\caption{{\bf Kirby moves from $\partial B_n$ to $\partial C_n$}}
\label{f:kirbybntocn}
\end{figure}

%\vspace{.1in}

For $n$ odd, the above identifications imply:
\begin{equation}
\frac{n^2-n}{2}\mu_1 \equiv \frac{n^2-n}{2}\lambda_1 \mod{n^2} \, .
\end{equation}

\noindent Giving us:
\begin{equation}
PD\Gamma_{(\partial B_n)}(\xi,\mathfrak{s}) \equiv PD\Gamma_{\partial C_n}(\xi',\mathfrak{r}) \, .
\end{equation}

%\vspace{.1in}

For $n$ even, we again match up the spin structures by following the spin structure $\mathfrak{s}_1$ through the Kirby moves in Figure~\ref{f:kirbybntocn}. We represent the spin structure $\mathfrak{s}_1$ by coloring the corresponding unknots in its characteristic sublink with a grey color. Thus, we can see that the spin structure $\mathfrak{s}_1$ of $\partial B_n$ is identified with the spin structure $\mathfrak{r}_1$ of $\partial C_n$. Similar to previous calculations, the relations in (\ref{eq:mulambda}) imply:
\begin{eqnarray}
\frac{2n^2-n}{2}\mu_1 &\equiv& \frac{n^2-n}{2}\lambda_1 \mod{n^2}  \\
\frac{n^2-n}{2}\mu_1 &\equiv& \frac{2n^2-n}{2}\lambda_1 \mod{n^2} \, .
\end{eqnarray}

%\vspace{.1in}

\noindent Therefore,
\begin{equation}
PD\Gamma_{(\partial B_n)}(\xi,\mathfrak{s}_1) \equiv PD\Gamma_{\partial C_n}(\xi',\mathfrak{r}_1)
\end{equation}
\begin{equation}
PD\Gamma_{(\partial B_n)}(\xi,\mathfrak{s}_1) \equiv PD\Gamma_{\partial C_n}(\xi',\mathfrak{r}_1) \, .
\end{equation} 
%\vspace{.1in}

\section{Appendices}

\appendix

\section{}
\label{a:appa}
The following sequence of Kirby diagrams show the equivalence of the two different Kirby diagrams for the rational homology balls $B_n$, as seen in Figures~\ref{f:bn} and \ref{f:bnsphm}.  We start off with Figure~\ref{f:App1}, a Kirby diagram of $B_n$ as in Figure~\ref{f:bn}, and illustrate the $n$ positive twists in Figure~\ref{f:App2}. Next we add a cancelling $1/2$-handle pair which includes a $0$-framed two-handle, Figure~\ref{f:App3}. After this, we slide the $(n-1)$-framed handle off of the $0$-framed handle, and obtain Figure~\ref{f:App4}, where the $(n-1)$-framed handle becomes a $(n-3)$-framed handle. We can continue to perform handleslides as seen in Figure~\ref{f:App5} and Figure~\ref{f:App6}, until we have completely slid off the original two-handle from the original one-handle, obtaining Figure~\ref{f:App7}. Finally, we remove a cancelling $1/2$-handle pair, and obtain Figure~\ref{f:App8}, with $n$ negative twists, which corresponds to Kirby diagram Figure~\ref{f:App9} (identical to Figure~\ref{f:bnsphm}).

\begin{figure}[ht]
\begin{minipage}[b]{0.45\linewidth}
\centering
\includegraphics[height=130pt, width=140pt]{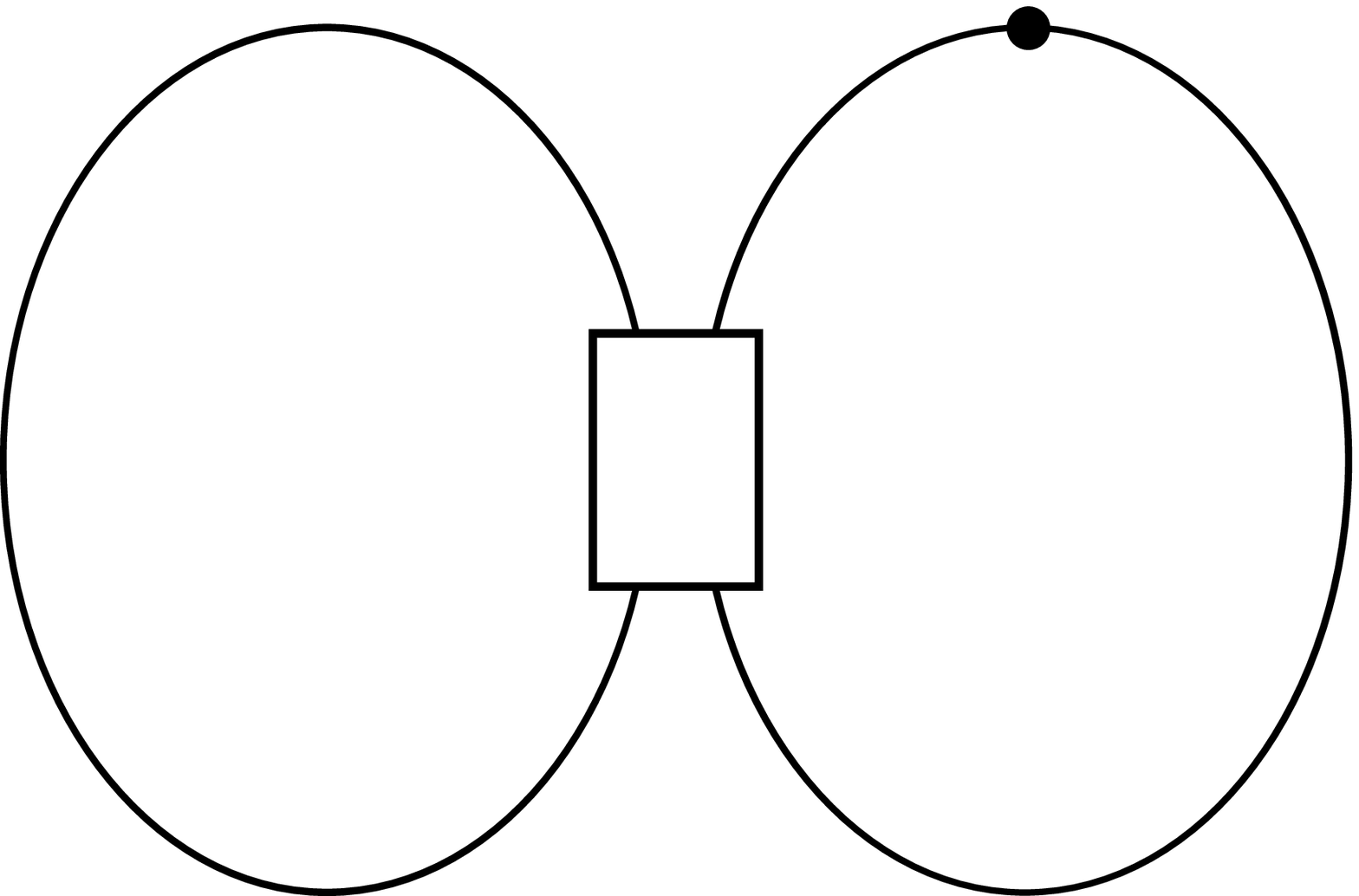}
\caption{ }
\labellist
\small\hair 2pt
\pinlabel $n$ at -195 160
\pinlabel $n-1$ at -260 310
\pinlabel $n$ at 70 150
\pinlabel $twists$ at 70 130
\pinlabel $n-1$ at 70 315
\endlabellist
\label{f:App1}
\end{minipage}
%\hspace{0.1cm}
\begin{minipage}[b]{0.45\linewidth}
\centering
\includegraphics[height=130pt, width=140pt]{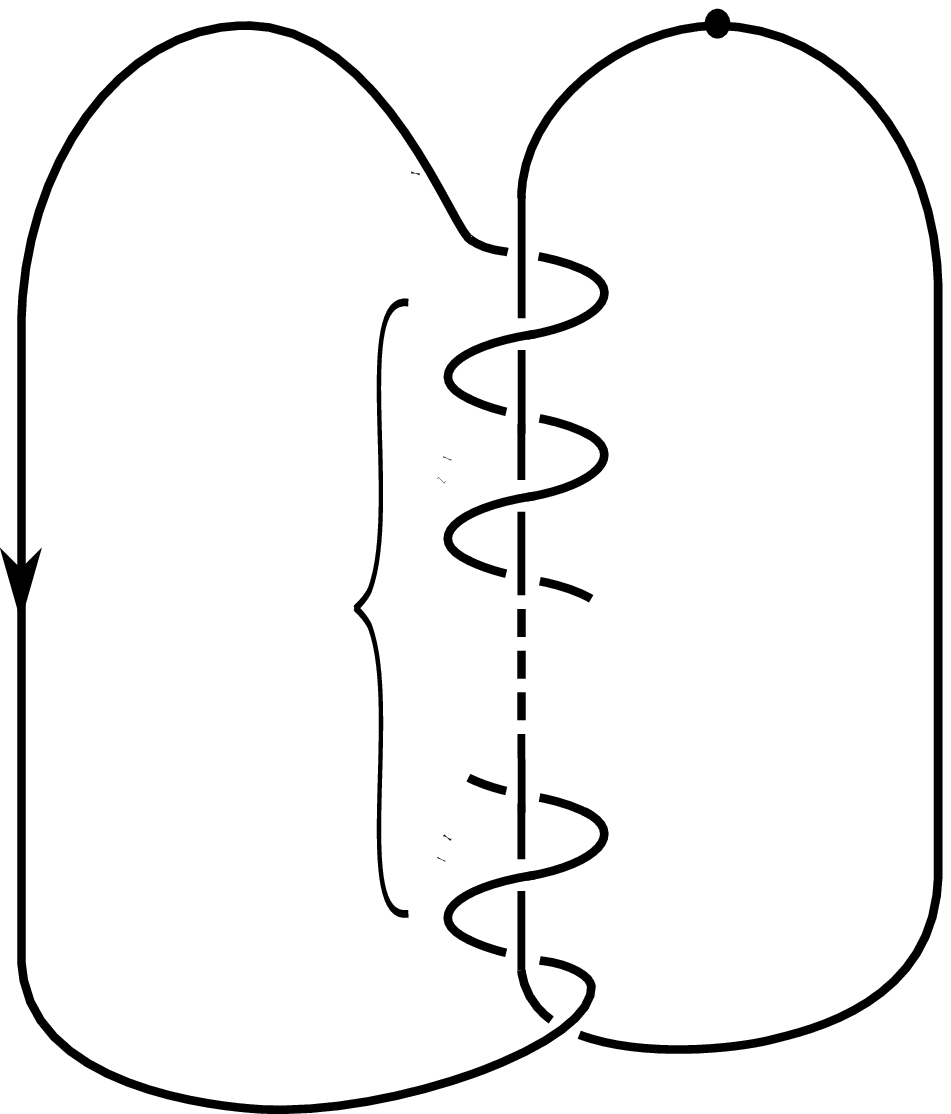}
\caption{ }
\label{f:App2}
\end{minipage}
\end{figure}

\begin{figure}[ht]
\begin{minipage}[b]{0.45\linewidth}
\centering
\includegraphics[height=130pt, width=140pt]{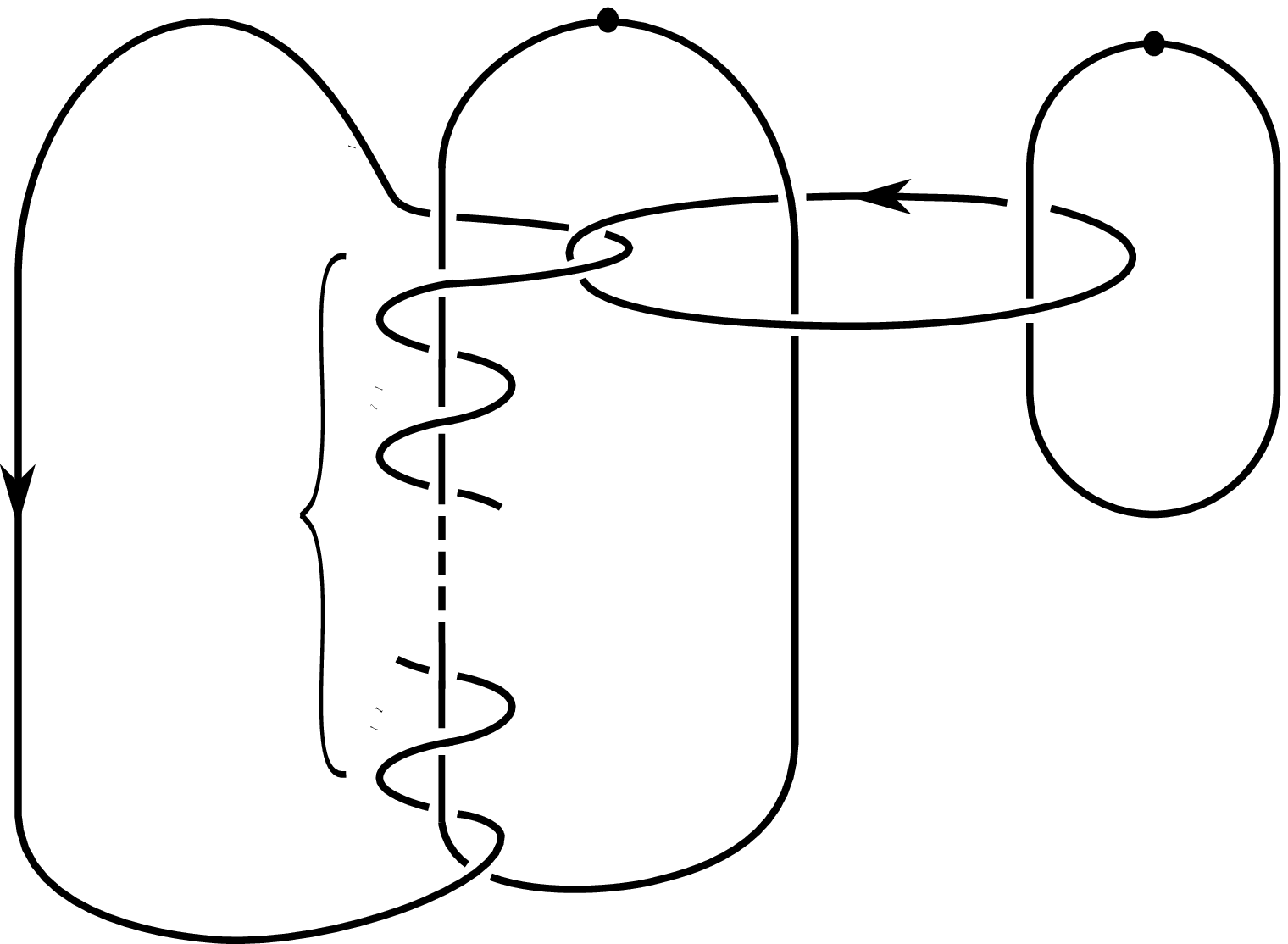}
\caption{ }
\labellist
\small\hair 2pt
\pinlabel $n-1$ at -335 352
\pinlabel $n$ at -345 165
\pinlabel $twists$ at -345 145
\pinlabel $0$ at -160 292
\pinlabel $n-3$ at 95 365
\pinlabel $n-1$ at 10 185
\pinlabel $twists$ at 10 165
\pinlabel $0$ at 200 287
\endlabellist
\label{f:App3}
\end{minipage}
%\hspace{0.1cm}
\begin{minipage}[b]{0.45\linewidth}
\centering
\includegraphics[height=130pt, width=140pt]{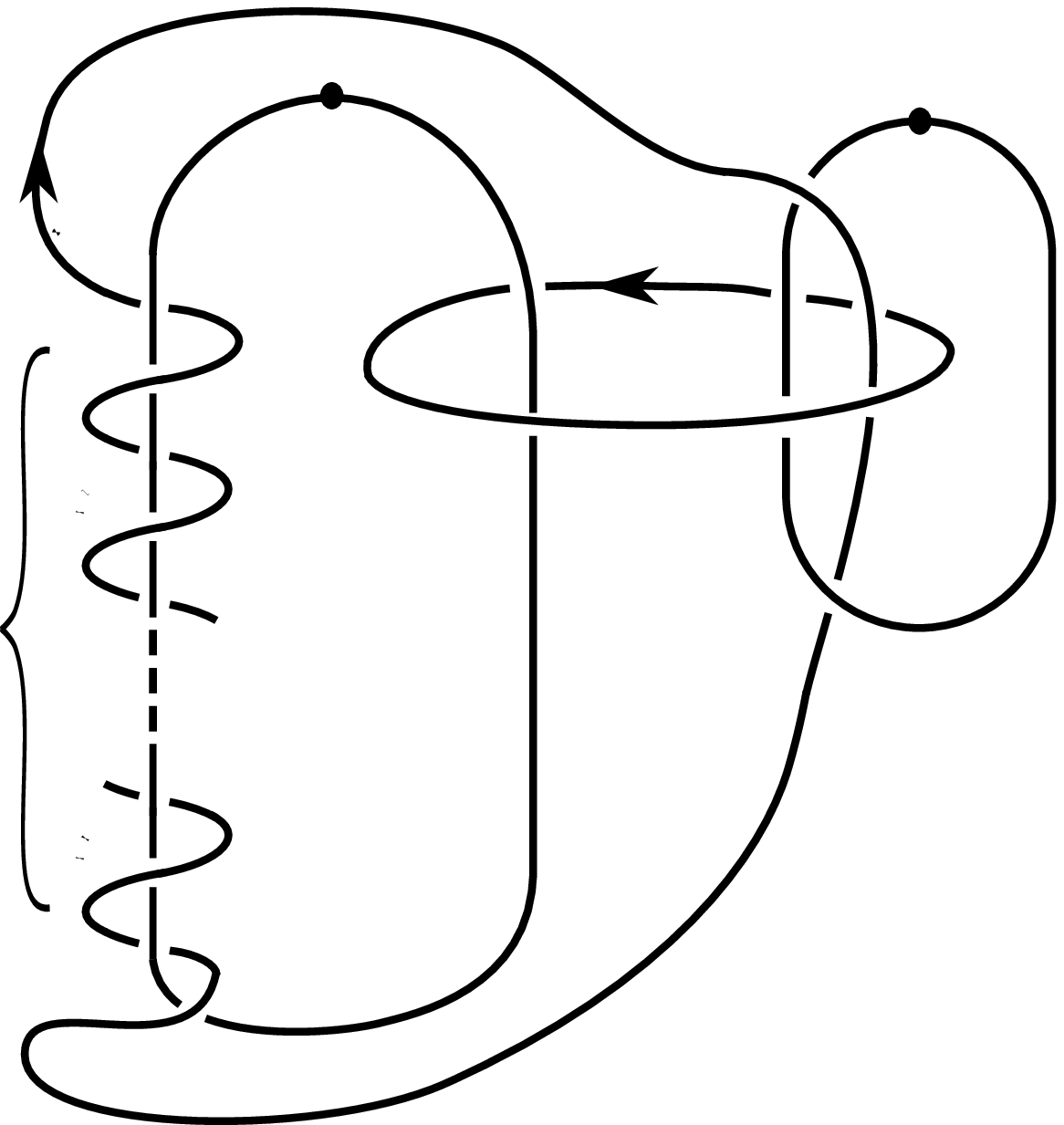}
\caption{ }
\label{f:App4}
\end{minipage}
\end{figure}

\begin{figure}[ht]
\begin{minipage}[b]{0.45\linewidth}
\centering
\includegraphics[height=130pt, width=140pt]{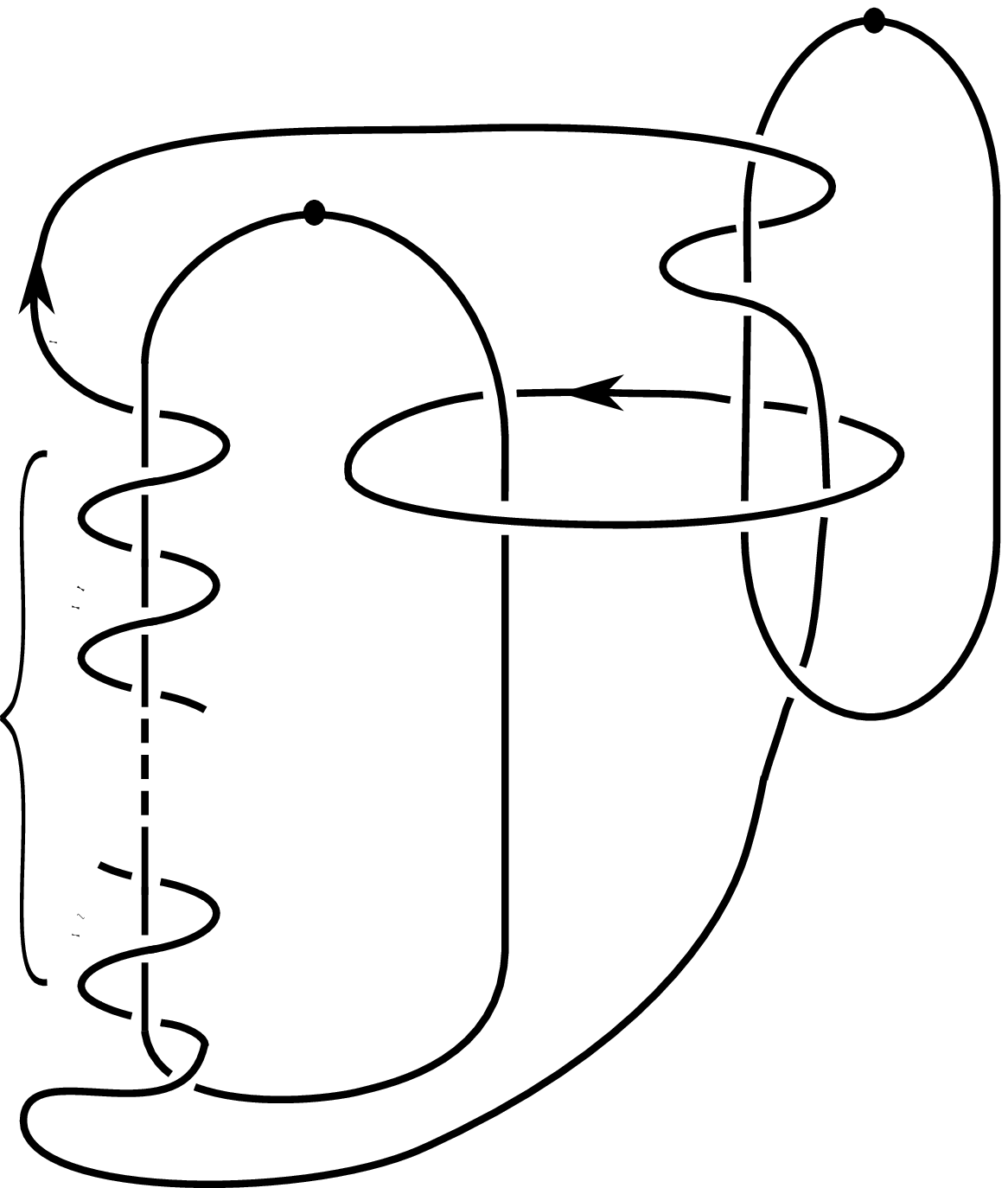}
\caption{ }
\labellist
\small\hair 2pt
\pinlabel $n-5$ at -305 540
\pinlabel $n-2$ at -455 260
\pinlabel $twists$ at -455 225
\pinlabel $0$ at -220 420
\pinlabel $-n+1$ at 75 440
\pinlabel $n-1$ at 160 385
\pinlabel $twists$ at 160 350
\pinlabel $0$ at 190 283
\endlabellist
\label{f:App5}
\end{minipage}
%\hspace{0.1cm}
\begin{minipage}[b]{0.45\linewidth}
\centering
\includegraphics[height=130pt, width=140pt]{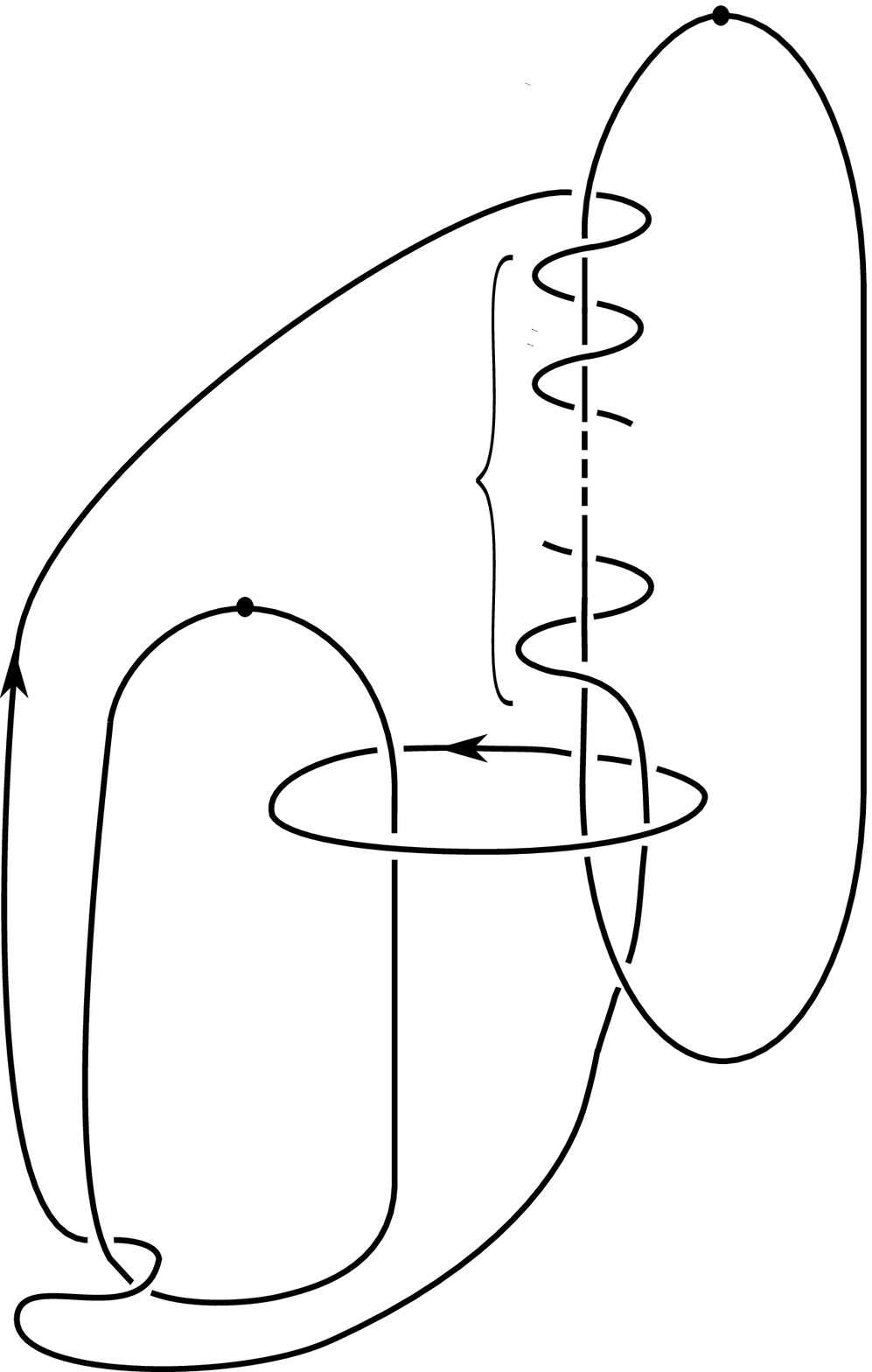}
\caption{ }
\label{f:App6}
\end{minipage}
\end{figure}

\begin{figure}[ht]
\begin{minipage}[b]{0.31\linewidth}
\centering
\includegraphics[scale=0.28]{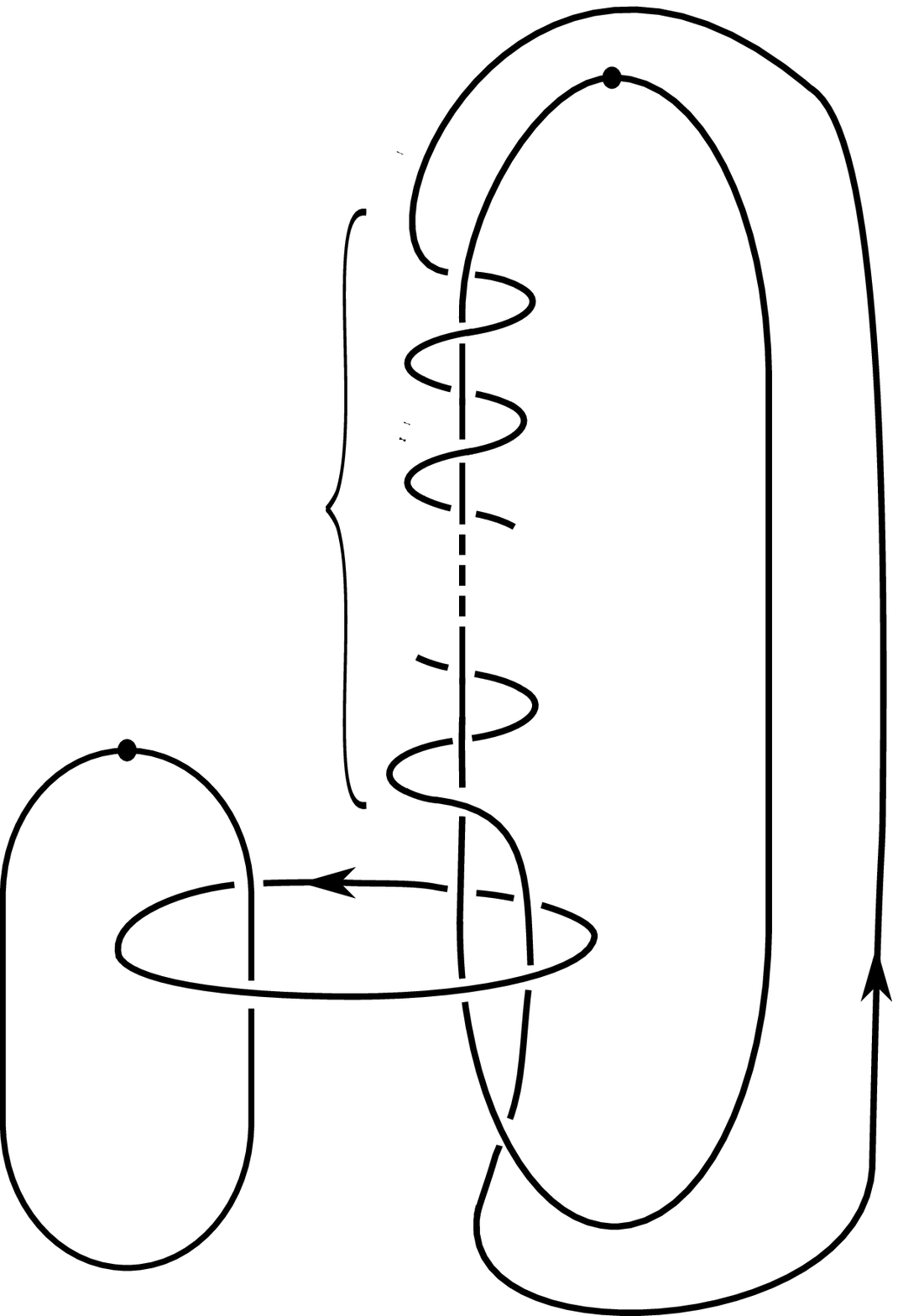}
\caption{ }
\labellist
\small\hair 2pt
\pinlabel $-n-1$ at -250 445
\pinlabel $n-1$ at -340 280
\pinlabel $twists$ at -340 255
\pinlabel $0$ at -260 90
\pinlabel $n$ at 57 170
\pinlabel $twists$ at 57 150
\pinlabel $-n-1$ at 75 317
\pinlabel $-n$ at 510 155
\pinlabel $-n-1$ at 430 295
\endlabellist
\label{f:App7}
\end{minipage}
%\hspace{0.1cm}
\begin{minipage}[b]{0.36\linewidth}
\centering
\includegraphics[scale=0.34]{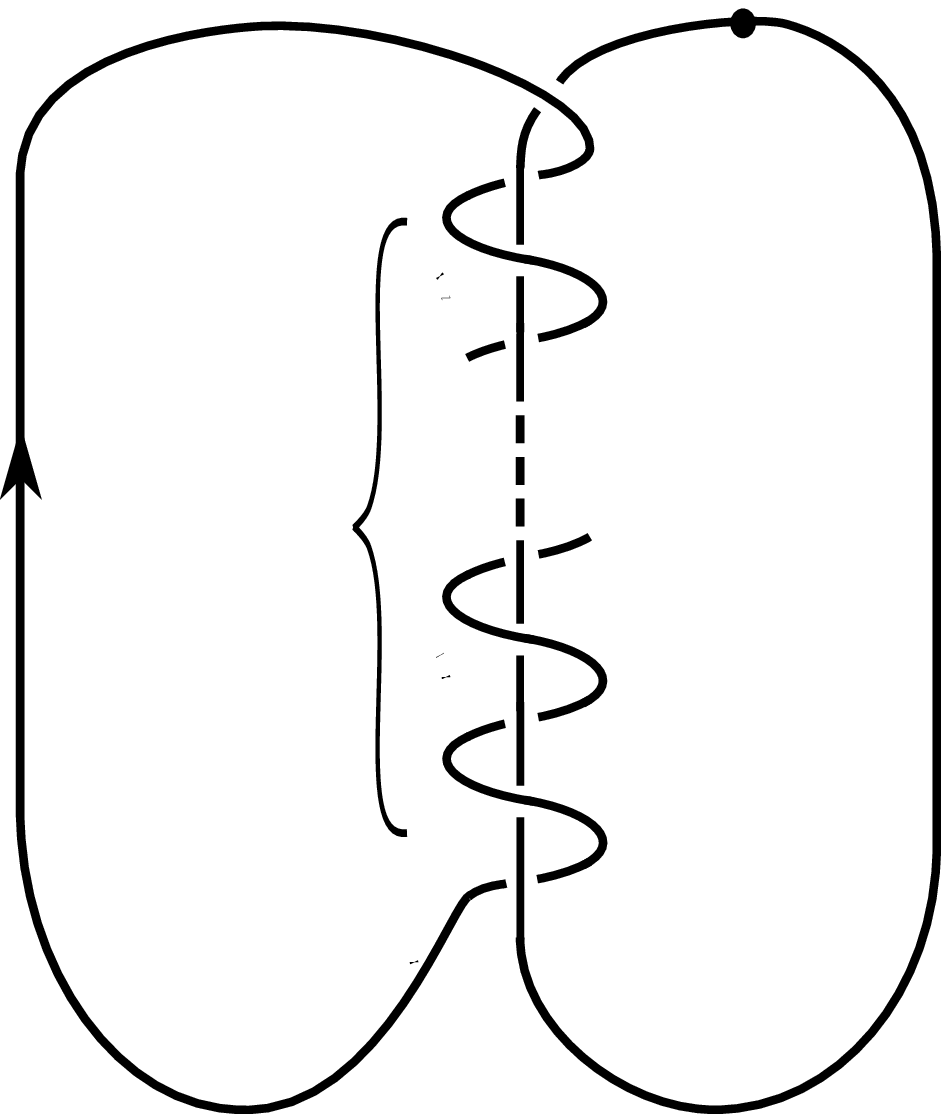}
\caption{ }
\label{f:App8}
\end{minipage}
%\hspace{0.1cm}
\begin{minipage}[b]{0.31\linewidth}
\centering
\includegraphics[height=100pt,width=120pt]{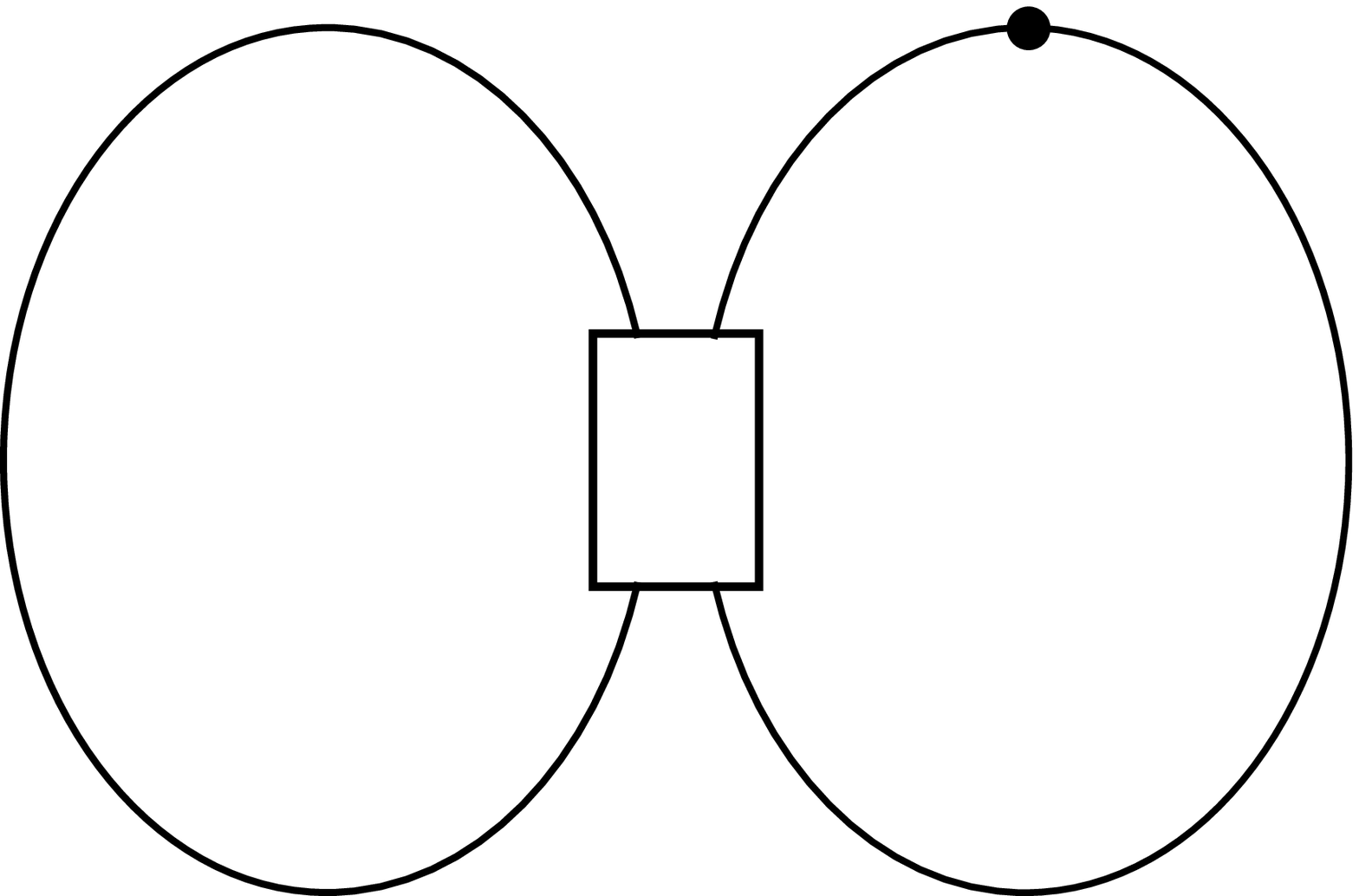}
\caption{ }
\label{f:App9}
\end{minipage}
\end{figure}

\end{document}